\documentclass[12pt,a4paper]{amsart}

\usepackage[leqno]{amsmath}
\usepackage{amsthm}
\usepackage{amsfonts}
\usepackage{amssymb}
\usepackage{eucal}
\usepackage[all]{xy}
\CompileMatrices

\usepackage{mathrsfs}

\usepackage{amsmath,amssymb,mathrsfs,xy,amsthm,bm,url}
\input{xy}
\xyoption{all}

\usepackage{xspace}
\usepackage{hyperref}

\theoremstyle{definition}
  \newtheorem{definition}[subsection]{Definition}
  
  \newtheorem{definition-proposition}[subsection]{Definition-Proposition}
  \newtheorem{example}[subsection]{Example}
  \newtheorem{remark}[subsection]{Remark}

\theoremstyle{theorem}
  \newtheorem{theorem}[subsection]{Theorem}
  \newtheorem{proposition}[subsection]{Proposition}
  \newtheorem*{proposition*}{Proposition}
  \newtheorem{lemma}[subsection]{Lemma}
  \newtheorem{corollary}[subsection]{Corollary}
  \newtheorem{conjecture}[subsection]{Conjecture}

  \newtheorem{assumption}[subsection]{Assumption}

\newcommand{\ncm}{\mathrm{ncm}}
\newcommand{\rhorig}{{$\rho$-rigid}\xspace}

\newcommand{\adele}{\mathbb{A}_\mathrm{f}}
\newcommand{\Abb}{\mathbb{A}}
\newcommand{\Cbb}{\mathbb{C}}
\newcommand{\Gbb}{\mathbb{G}}
\newcommand{\Nbb}{\mathbb{N}}
\newcommand{\Qbb}{\mathbb{Q}}
\newcommand{\Rbb}{\mathbb{R}}
\newcommand{\Sbb}{\mathbb{S}}
\newcommand{\Tbb}{\mathbb{T}}
\newcommand{\Zbb}{\mathbb{Z}}
\newcommand{\Zbhat}{\hat{\mathbb{Z}}}

\newcommand{\Cbf}{\mathbf{C}}
\newcommand{\Hbf}{\mathbf{H}}
\newcommand{\Gbf}{\mathbf{G}}
\newcommand{\ibf}{\mathbf{i}}
\newcommand{\Lbf}{\mathbf{L}}
\newcommand{\Mbf}{\mathbf{M}}
\newcommand{\Nbf}{\mathbf{N}}
\newcommand{\Pbf}{\mathbf{P}}

\newcommand{\Tbf}{\mathbf{T}}
\newcommand{\Ubf}{\mathbf{U}}
\newcommand{\Vbf}{\mathbf{V}}
\newcommand{\Wbf}{\mathbf{W}}
\newcommand{\Xbf}{\mathbf{X}}

\newcommand{\GLbf}{\mathbf{GL}}

\newcommand{\Sfrak}{\mathfrak{S}}
\newcommand{\Tfrak}{\mathfrak{T}}
\newcommand{\Xfrak}{\mathfrak{X}}
\newcommand{\Yfrak}{\mathfrak{Y}}

\newcommand{\mrm}{\mathrm{m}}

\newcommand{\Hscr}{\mathscr{H}}
\newcommand{\Pscr}{\mathscr{P}}
\newcommand{\Sscr}{\mathscr{S}}

\newcommand{\Fcal}{\mathcal{F}}
\newcommand{\Pcal}{\mathcal{P}}

\newcommand{\Ucal}{\mathcal{U}}

\newcommand{\Ycal}{\mathcal{Y}}

\newcommand{\ab}{\mathrm{ab}}
\newcommand{\der}{\mathrm{der}}

\newcommand{\ra}{\rightarrow}
\newcommand{\lra}{\longrightarrow}

\newcommand{\mono}{\hookrightarrow}
\newcommand{\isom}{\cong}

\newcommand{\Hom}{\mathrm{Hom}}
\newcommand{\Group}{\mathrm{Group}}
\newcommand{\GSp}{\mathrm{GSp}}
\newcommand{\Sp}{\mathrm{Sp}}
\newcommand{\Ker}{\mathrm{Ker}}

\newcommand{\Nm}{\mathrm{Nm}}
\newcommand{\Res}{\mathrm{Res}}
\newcommand{\Supp}{\mathrm{Supp}}

\newcommand{\spl}{\mathrm{spl}}

\newcommand{\bsh}{\backslash}
\newcommand{\inv}{{-1}}
\newcommand{\ot}{\overset}
\newcommand{\pr}{\mathrm{pr}}
\newcommand{\id}{\mathrm{id}}

\newcommand{\wrt}{{with\ respect\ to}\xspace}
\newcommand{\ifof}{{if\ and\ only\ if}\xspace}
\newcommand{\cosg}{{compact\ open\ subgroup}\xspace}
\newcommand{\cosgs}{{compact\ open\ subgroups}\xspace}

\newcommand{\Qac}{{\overline{\mathbb{Q}}}}
\newcommand{\Gal}{\mathrm{Gal}}
\newcommand{\Gr}{\mathrm{Gr}}
\newcommand{\Cok}{\mathrm{Cok}}

\newcommand{\pitilde}{\tilde{\pi}}
\newcommand{\rec}{\mathrm{rec}}
\newcommand{\limproj}{\varprojlim}
\newcommand{\limind}{\varinjlim}
\newcommand{\Aut}{\mathrm{Aut}}
\newcommand{\rig}{\mathrm{rig}}
\newcommand{\cm}{\mathrm{cm}}
\newcommand{\Oscr}{\mathscr{O}}
\newcommand{\Acal}{\mathcal{A}}
\newcommand{\GL}{\mathrm{GL}}
\newcommand{\tor}{\mathrm{tor}}

\title{On Special Subvarieties of Kuga Varieties II}

\author{Ke Chen}

\address{Institut f\"ur Mathematik, Johannes Gutenberg Universit\"at Mainz, 55099 Mainz, Deutchland }
\email{chenk@uni-mainz.de}



\begin{document}

\begin{abstract}  In this paper we prove the equidistribution of $\Cbf$-special subvarieties in certain Kuga varieties, which implies a special case of the general Andr\'e-Oort conjecture formulated for mixed Shimura varieties proposed by R.Pink. The main idea is to reduce the equidistribution to a theorem of Szpiro-Ullmo-Zhang on small points of abelian varieties and a theorem on the equiditribution of $C$-special subvarieties of Kuga varieties of rigid type treated by the author in a previous paper.
\end{abstract}

\maketitle
\tableofcontents
\section*{Introduction}

As is pointed out in \cite{pink-combination}, the Andr\'e-Oort conjecture for mixed Shimura varieties provides a common generalization of the Manin-Mumford conjecture for abelian varieties and the Andr\'e-Oort conjecture for pure Shimura varieties. The present paper is aimed to prove a special case of the equidistribution conjecture, which refines the Andr\'e-Oort conjecture with measure-theoretic tools. Recall that (cf.\cite{chen-rigid} 1.9 (2)) a connected Kuga variety carries a canonical probability measure, and so it is with its special subvarieties. The equidistribution conjecture for Kuga varieties is formulated as:

\begin{conjecture}\label{equidistribution-conjecture} Let $M$ be a Kuga variety, and $(M_n)_n$ a sequence of special subvarieties in $M$, strict in the sense that for any $M'\subsetneq M$ special subvariety, $M_n\nsubseteq M$ for $n$ large enough. Fix $F\subset\Cbb$ a number field over which $M$ admits a canonical model, and for each $n$ put $\Oscr_n$ the $\Gal(\Qac/F)$-orbit of $M_n$ with $$\mu_n:=\dfrac{1}{\#\Oscr_n}\sum_{Z\in\Oscr_n}\mu_Z$$ with $\mu_Z$ the canonical probability measure on $Z(\Cbb)$. Then $(\mu_n)_n$ converges to the canonical probability measure on $M(\Cbb)$ for the weak-* topology.
\end{conjecture}

The prototype of Kuga varieties is the pull-back of the universal family $\Ucal_g$ of abelian varieties over $\Acal_g$ a Siegel moduli variety (with suitable level structure) along a morphism of pure Shimura varieties $S\ra \Acal_g$. When we pull back the universal family to a special point (i.e. CM point) of $\Acal_g$, the resulting Kuga variety is nothing but the corresponding CM abelian variety. In this case the equidistribution conjecture is known, because special subvarieties in such CM abelian varieties are the same as torsion subvarieties, i.e. translation of abelian subvarieties by torsion points, and even the assumption of CM type is not necessary:

\begin{theorem}\label{ullmo-icm}

Let $F\subset\Cbb$ be a number field, and $A$ an abelian variety over $F$. Take $(A_n)_n$ a sequence of torsion subvarieties, strict in the sense that for any torsion subvariety $A'\subsetneq A$, we have $A_n\nsubseteq A'$ for $n$ large enough. Put $\Oscr_n$ the $\Gal(\Qac/F)$-orbit of $A_n$, and $\mu_n$ the average over $\Oscr_n$ of the canonical measures of $Z$ in $\Oscr_n$. Then $\mu_n$ converges to the canonical measure on $A(\Cbb)$ for the weak-* topology.
\end{theorem}
Here the canonical measure for a complex abelian variety is just the Haar measure on the underlying compact tori, and the canonical measure for a torsion subvariety is obtained by translation from the Haar measure on the corresponding abelian subvariety.

On the other hand, in a previous paper \cite{chen-rigid} we have prove the weak-* compactness of canonical measures associated to \rhorig $\Cbf$-sepcial subvarieties of Kuga varieties. The notion of $\rho$-rigidity is to rule out the case of CM abelian variety as much as possible. In the present paper we combine these two results into:

\begin{theorem}[cf.\ref{main-theorem} Main theorem]\label{statement-main-theorem} Let $M$ be a connected Kuga variety defined by some Kuga datum $(\Pbf,Y)=\Vbf\rtimes(\Gbf,X)$ with its canonical model over some number field $F\subset\Cbb$, which is isogeneous to a product of the form $M^\cm\times M^\rig \times S'$ where

(i) $M^\cm$ is a CM abelian variety;

(ii) $M^\rig$ a conencted Kuga variety,  defined by some datum $(\Pbf^\rig,Y^\rig)=\Vbf^\rig\rtimes_\rho(\Gbf^\rig,X^\rig)$ such that the action of $\Gbf^\rig$ on $\Vbf^\rig$ is rigid and faithful, and the connected center of $\Gbf^\rig$ splits over $\Qbb$;

(iii) $S'$ a connected pure Shimura variety.

Write $\Cbf$ for the connected center of $\Gbf$, and assume that the connected center $\Cbf^\rig$ of $\Gbf^\rig$ splits over $\Qbb$.

Let $(M_n)_n$ be a sequence of $\Cbf$-special subvarieties of $M$, which is $\Cbf$-strict in the sense for any $\Cbf$-special subvariety $M'\subsetneq M$ we have $M_n\nsubseteq M'$ for $n$ sufficiently large. Then the averages of the canonical measures over the $\Gal(\Qac/F)$-orbit of the $M_n$'s converge to the canonical measure on $M(\Cbb)$ for the weak-* topology.
\end{theorem}

In the paper the restriction  on $M^\rig$ (including the condition on $\Cbf^\rig$) is referred to as strong rigidity. The motivation is an result of K.Ribet, which asserts that for an abelian variety $A$ defined over some number field $F$, the torsion subgroup of $A(F\Qbb^\ab)$ is always finite. From this we show that the Galois conjugates of special subvarieties in $M^\cm$ and in $M^\rig\times S'$ are independent of each other, and the approximation to the canonical measure on $M(\Cbb)$ is reduced to the approximation on each factor.

We also prove that the main theorem holds when the factor $M^\cm$ is trivial. However we are not yet able to treat the remaining case when $M^\cm$ appears with $M^\rig$ not strong, i.e. $\Cbf^\rig$ non-split over $\Qbb$.
%second version

The paper is organized as follows:

(1) In Section 1 we recall the basic definition of Kuga data, as was presented in \cite{chen-rigid} Section 1. We also collect materials about the canonical models of Kuga varieties, in particular the Galois action on the set of connected components of a Kuga variety.

(2) In Section 2 we show that up to shrinking the level a Kuga variety $M$ can be decomposed into a direct product of Kuga varieties $M=M^\cm\times M^\rig\times S'$ where $M^\cm$ is a CM abelian variety, $M^\rig$ is a rigid Kuga variety given by a rigid Kuga datum $(\Pbf^\rig,Y^\rig)=\Vbf^\rig\rtimes(\Gbf^\rig,X^\rig)$ with $\Gbf^\rig$ acting on $\Vbf^\rig$ faithfully, and $S'$ is a pure Shimura variety. The main theorem is stated, from which is deduced the main corollary on a special case of the Andr\'e-Oort conjecture for Kuga varieties.

(3) In Section 3 we show that the main theorem holds for rigid Kuga varieties, which is an immediate consequence to the main result of \cite{chen-rigid} through the functorial properties of conjugates of Shimura varieties proved in \cite{milne-shih} and \cite{milne-conjugate}.

(4) In Section 4 we show that in the main theorem the Galois orbit of a $\Cbf$-special subvariety  in $M$ is, up to uniformly bounded constants, of the same size as the product of the Galois orbits in the factors, hence reduce the main theorem to known cases on the factors. The independence is deduced from a result of Ribet on the finiteness of torsion points of an abelian variety in a cyclotomic tower.

(5) Finally in the appendix we sketch some facts about the construction and functorial properties of conjugates of Kuga varieties by elements in $\Aut(\Cbb/\Qbb)$, which are parallel to the pure case treated by Milne et.al.
%The aim of this paper is to prove the equidistribution of $\Cbf$-special subvarieties of Kuga varieties. In \cite{chen-rigid} we have shown the equidistribution under the auxiliary assumption of $\rho$-rigidity. In this paper we treat the general case  by combining the result in \cite{chen-rigid} with the above theorem for CM abelian varieties.

The treatment in this paper requires little knowledge of the ergodic-theoretic arguments, such as lattice spaces and measures on them, used in \cite{chen-rigid}. %The essential ingredient is the reduction through functorial properties of canonical models of Shimura varieties and the finiteness of $A(F\Qbb^\ab)$ for $A$ a CM abelian variety over $F$.

\subsection*{Acknowledgement} This paper grows out of the thesis of the author. The author thanks his advisor Prof.Emmanuel ULLMO for guiding him into this subject with great care and patience, without which the paper wouldn't have existed. He also thanks Prof.M\"uller-Stach and Prof.Zuo for their kind help. The work is partially supported by the project SFB/TR45 ''Periods, moduli spaces, and arithmetic of algebraic varieties''.% during his stay in Johannes Gutenberg Universit\"at Mainz.

\subsection*{Convention} Over a base field $k$, a linear $k$-group $\Hbf$ is understood to be a smooth affine algebraic $k$-group.  A $k$-factor of a reductive $k$-group is understood to be a minimal normal non-commutative $k$-subgroup (of dimension $>0$), i.e. a non-commutative $k$-simple normal $k$-subgroup.

$\Sbb$ stands for the Deligne torus $\Res_{\Cbb/\Rbb}\Gbb_\mrm$. $\ibf$ is a fixed square root of -1 in $\Cbb$.

Number fields are always understood to be embedded in $\Cbb$. Denote by $\Qac$ the algebraic closure of $\Qbb$ in $\Cbb$, and $\adele$  the set of finite adeles.
%For $\Hbf$ a linear $\Qbb$-group, set $\Xfrak(\Hbf):=\Hom_{\Gr/\Rbb}(\Sbb,\Hbf_\Rbb)$  on which $\Hbf(\Rbb)$  acts by conjugation from the left.

A linear $\Qbb$-group $\Hbf$ is compact if $\Hbf(\Rbb)$ is a compact Lie group. Upper and lower scripts ${}^\circ$, ${}^+$, ${}_+$ follow the convention of P.Delinge in \cite{deligne-pspm}.%(resp. $\Yfrak(\Hbf):=\Hom_{\Gr/\Cbb}(\Sbb_\Cbb,\Hbf_\Cbb)$)(resp. $\Hbf(\Cbb)$)

For a real or complex manifold, its archimedean topology is the one locally deduced from the archimedean metric (i.e.  the usual absolute value) on $\Rbb^n$ or $\Cbf^m$.

\section{Preliminaries on Kuga varieties}

In this section we recall necessary facts about the canonical models of Kuga varieties, and give some construction that will be used in the main result. We refer to \cite{chen-rigid} Section 1 for generalities on Kuga varieties. It suffices to keep in mind the following reconstruction (cf.\cite{chen-rigid} 1.3, 1.7):

\begin{definition}\label{kuga-datum}Recall that for $\Pbf$ a linear $\Qbb$-group we put $\Xfrak(\Pbf):=\Hom_{\Gr/\Rbb}(\Sbb,\Pbf_\Rbb)$ on which $\Pbf(\Rbb)$ acts from the left by conjugation, with $\Sbb=\Res_{\Cbb/\Rbb}\Gbb_\mrm$ the Deligne torus.

(1) A Kuga datum is a pair $(\Pbf,Y)$ constructed as follows:

(1-a) a pure Shimura datum $(\Gbf,X)$ in the sense of Deligne \cite{deligne-pspm} 2.1.1;

(1-b) a finite-dimensional algebraic representation $\Gbf\ra\GLbf_\Qbb(\Vbf)$, such that for any $x\in X$, the composition $\rho\circ x:\Sbb\ra\GLbf_\Rbb(\Vbf_\Rbb)$ defines a rational Hodge structure of type $\{(-1,0),(0,-1)\}$, and that the connected center $\Cbf_\Gbf$ of $\Gbf$ acts on $\Vbf$ through a $\Qbb$-torus isogeneous to a product of a split $\Qbb$-torus and a compact $\Qbb$-torus.

(1-c) $\Pbf=\Vbf\rtimes_\rho\Gbf$, and $Y$ equals the $\Pbf(\Rbb)$-orbit of $X$ in $\Xfrak(\Pbf)$ (which is well-defined because of the inclusion chain $X\subset\Xfrak(\Gbf)\subset\Xfrak(\Pbf)$ by $\Gbf\subset\Pbf$); we also require $\Pbf$ to be minimal, in the sense that for any $\Qbb$-subgroup $\Hbf\subsetneq\Pbf$, there exists $y$ such that $y(\Sbb)\nsubseteq\Hbf_\Rbb$, and this is equivalent to the minimality of $\Gbf$ \wrt $X$.

Consequently, the push-forward $\pi_*:\Xfrak(\Pbf)\ra\Xfrak(\Gbf)$ induced by the projection $\pi:\Pbf\ra\Gbf$ induces a $\Vbf(\Rbb)$-torsor $\pi_*:Y\ra X$, with section $X\ra Y$ given by $\Gbf\mono\Pbf$. Note that $\pi_*:Y\ra X$ is a smooth surjective morphism of complex manifolds, equivariant \wrt $\pi:\Pbf(\Rbb)\ra\Gbf(\Rbb)$, and the fiber $\pi^\inv(x)$ is $\Vbf(\Rbb)$ with complex structure given by $\rho\circ x$.

We often write $(\Pbf,Y)=\Vbf\rtimes_\rho(\Gbf,X)$ to indicate that the Kuga datum is constructed out of the pure Shimura datum $(\Gbf,X)$ plus the representation $(\Vbf,\rho)$.

It should be pointed out that in \cite{chen-rigid} we have used pure Shimura data in the sense of Pink, where instead of $X\subset\Xfrak(\Gbf)$ we put a $\Gbf(\Rbb)$-equivariant map $h:X\ra\Xfrak(\Gbf)$ with finite fibers. By \cite{pink-thesis} 2.12, $h$ is injective on each connected component. Since we mainly work with connected Kuga varieties in this paper, one may restrict to pure Shimura data in the sense of Deligne.

(2) Morphisms of Kuga data are of the form $(f,f_*):(\Pbf,Y)\ra(\Pbf',Y')$, which are evidently defined as a homomorphism of $\Qbb$-group $f:\Pbf\ra\Pbf'$ plus $f_*:Y\ra Y'$ induced by $f_*:\Xfrak(\Pbf)\ra \Xfrak(\Pbf')$. The map $f_*:Y\ra Y'$ is a holomorphic map between complex manifolds, and is equivariant \wrt $f:\Pbf(\Rbb)\ra\Pbf'(\Rbb)$.

Subdata are those given by $(f,f_*)$ with $f$ an inclusion of $\Qbb$-subgroup and $f_*$ injective. The notion of product of two Kuga data $(\Pbf_i,Y_i)$ $(i=1,2)$ is $(\Pbf_1\times\Pbf_2,Y_1\times Y_2)$ defined in the evident way. And the notion of quotient of $(\Pbf,Y)$ by a normal $\Qbb$-subgroup $\Nbf$ is well-defined; when $\Nbf$ is unipotent, it is simply the usual quotient space $Y/\Nbf(\Rbb)$, cf.\cite{chen-rigid} 1.1(3).

(3) In particular, for $(\Pbf,Y)$ constructed from $(\Gbf,X)$ and $\rho$ as in (1), we have the canonical projection $\pi:(\Pbf,Y)\ra(\Gbf,X)$, which is the same as taking quotient by the unipotent radical $\Vbf$. We thus call $(\Gbf,X)$ the pure base of $(\Pbf,Y)$.  $\pi$ admits a pure section $i:(\Gbf,X)\mono(\Pbf,Y)$ induced by the inclusion $i:\Gbf\mono\Pbf$, which is a maximal pure subdatum of $(\Pbf,Y)$. All the maximal pure subdata of $(\Pbf,Y)$ are of the form $(v\Gbf v^\inv;v\rtimes X)$, with $v$ running through $\Vbf(\Qbb)$, due to the uniqueness of Levi $\Qbb$-subgroup up to conjugacy by  the unipotent radical.

An immediate corollary, cf.\cite{chen-rigid} 1.8, is that any Kuga subdatum $(\Pbf',Y')$ of the $(\Pbf,Y)$ (constructed from $(\Gbf,X)$ and $\rho$) is of the form $(\Pbf'=\Vbf'\rtimes v\Gbf'v^\inv,Y'=(\Vbf'(\Rbb)+v)\rtimes X')$, where $(\Gbf',X')$ is a pure subdatum of $(\Gbf,X)$, $\Vbf'$ is a subrepresentation of $\rho_{|_{\Gbf'}}:\Gbf'\ra\GLbf_\Qbb(\Vbf)$, and $v\in\Vbf(\Qbb)$ is unique up to $\Vbf'(\Qbb)$-translation. Here the formula of $Y'$ means that $Y'$ is the orbit of $X'$ in $Y$ under translation by the set $\Vbf'(\Rbb)+v$.

(4) We also talk about connected Kuga data, which are of the form $(\Pbf,Y;Y^+)$ with $(\Pbf,Y)$ a Kuga datum and $Y^+$ a connected component of $Y$. Notions such as morphisms of connected Kuga subdata, subdata, pure base, etc. are understood in the natural way.

%As we mainly work with connected Kuga varieties, it does not matter whether pure Shimura data are understood in the sens of Deligne \cite{deligne-pspm} or in the sense of Pink \cite{pink-thesis}, as they define the same category of connected pure Shimura varieties (hence the same category of connected Kuga varieties).

(5) For convenience, we say a Kuga datum $(\Pbf,Y)$ is toric if $\Pbf$ is a $\Qbb$-torus, whence $Y$ is a single point. In the literature they are usually called ''special subdata'', but we reserve the adjective ''special'' for special subvarieties, which in general are given by non-toric subdata.
\end{definition}

\begin{definition}\label{reflex-field}[cf.\cite{deligne-pspm} 2.2.1 and \cite{pink-thesis} 11.1]
Let $(\Pbf,Y)$ be a Kuga datum.

(1) The reflex field $E(\Pbf,Y)$ is the field of definition of the $\Pbf(\Cbb)$-conjugacy class of $\mu_y:\Gbb_{\mrm,\Cbb}\ra\Pbf_\Cbb$ ($y\in Y$ an arbitrary point) in   $\Xbf_\Pbf^\vee(\Cbb)/\Pbf(\Cbb)$ with $\Xbf_\Pbf^\vee(\Cbb):=\Hom_{\Gr/\Cbb}(\Gbb_{\mrm\Cbb},\Pbf_\Cbb)$ the set of cocharacters over $\Cbb$. Here by field of definition, we mean the smallest subfield $E$ of $\Cbb$, such that $\Aut(\Cbb/E)$ equals the isotropy group of the class of $\mu_y$ for the action of $\Aut(\Cbb/\Qbb)$ on $\Xbf_\Pbf^\vee(\Cbb)//\Pbf(\Cbb)$.

The following facts are standard:

(a) Once there is a morphism $(\Pbf,Y)\ra(\Pbf',Y')$, we then have $E(\Pbf,Y)\supset E(\Pbf',Y')$; in particular, if $(\Gbf,X)$ is a pure section of $(\Pbf,Y)$, then $E(\Gbf,X)=E(\Pbf,Y)$. Moreover if $(\Pbf,Y)=(\Pbf_1,Y_1)\times(\Pbf_2,Y_2)$ is a product of Kuga data, then $E(\Pbf,Y)=E_1E_2$ is the composition field of $E_i=E(\Pbf_i,Y_i)$ (which is argued through $\Xbf_{\Pbf_1\times\Pbf_2}^\vee=\Xbf_{\Pbf_1}^\vee\times\Xbf_{\Pbf_2}^\vee$).

(b) If $(\Tbf,x)$ is a toric Shimura datum, then the reflex field $E=E(\Tbf,x)$ is just the number field of definition for the homomorphism $\mu_x:\Gbb_{\mrm\Cbb}\ra\Tbf_\Cbb$. In particular it is contained in the splitting field of $\Tbf$.

And for $\mu_x:\Gbb_{\mrm E}\ra\Tbf_E$, we put $r_x$ to be the following composition $$\Res_{E/\Qbb}\Gbb_{\mrm E}\ot{\mu_x}\lra\Res_{E/\Qbb}\Tbf_E\ot{\Nm_{E/\Qbb}}\lra\Tbf.$$

(c) Reflex fields are number fields embedded in $\Cbb$. In fact, if $(\Gbf,X)$ is a pure non-toric Shimura datum, with $x\in X$ a special point, sending $\Sbb$ into $\Hbf_\Rbb$ for some maximal $\Qbb$-torus $\Hbf\subset\Gbf$, then we have the equality $$\Xbf_\Gbf^\vee(\Cbb)/\Gbf(\Cbb)=\Xbf_\Hbf^\vee(\Cbb)/W(\Gbf_\Cbb,\Hbf_\Cbb)$$ where $W(\Gbf_\Cbb,\Hbf_\Cbb)$ is the Weyl group of $\Hbf_\Cbb$, and the $\Gbf(\Cbb)$-conjugacy class of $\mu_x$ on the left hand side is the same as the $W(\Gbf_\Cbb,\Hbf_\Cbb)$-conjugacy class of $\mu_x:\Gbb_\mrm\ra\Hbf_\Cbb$. Now that $\Hbf$ is defined over $\Qbb$, $\mu_x$  descends over some number field in $\Cbb$ (contained in the splitting field of $\Hbf$).

\end{definition}

We then define Kuga varieties and special subvarieties:

\begin{definition}\label{kuga-variety}

(1) Let $(\Pbf,Y)$ be a Kuga datum, and $K\subset\Pbf(\adele)$ a \cosg. The Kuga variety defined by $(\Pbf,Y)$ at level $K$ is a reduced algebraic variety $M_K(\Pbf,Y)$ over $E(\Pbf,Y)$ whose complex points are given as $$M_K(\Pbf,Y)(\Cbb)=\Pbf(\Qbb)\bsh[Y\times\Pbf(\adele)/K]$$ with $\Pbf(\Qbb)$ acting on the product $Y\times\Pbf(\adele)/K$ along the diagonal.

The fact that the quotient space above underlies an algebraic variety is proved in \cite{pink-thesis} Chapter 9, generalizing the Baily-Borel compactification in the pure case. It is smooth if the level $K$ is neat. The fact that this complex algebraic variety is defined over $\Qac$, and further admits a canonical model over the reflex field, is established through a series of works of P.Deligne, J.Milne, M.Borovoi, R.Pink, etc.

Note that if one fixes $Y^+$ a connected component of $Y$, then $\Pbf(\Qbb)_+$ equals the stabilizer of $Y^+$ in $\Pbf(\Qbb)$, and by taking $\{g\}$ a (finite) set of representatives for the double quotient $\Pbf(\Qbb)_+\bsh\Pbf(\adele)/K$, one may rewrite $$M_K(\Pbf,Y)(\Cbb)=\coprod_g\Gamma_K(g)\bsh Y^+$$ with $\Gamma_K(g):=\Pbf(\Qbb)_+\cap K$ a congruence subgroup stabilizing $Y^+$. The set of geometric connected component of $M_K(\Pbf,Y)$ is thus $\Pbf(\Qbb)_+\bsh\Pbf(\adele)/K$.

(2) When we are given a morphism of Kuga data $(f,f_*):(\Pbf,Y)\ra(\Pbf',Y')$ with $K\subset\Pbf(\adele)$ and $K'\subset\Pbf'(\adele)$ \cosgs such that $f(K)\subset K'$, we get a morphism of Kuga varieties $f:M_K(\Pbf,Y)\ra M_{K'}(\Pbf',Y')$, which is the evident map $\Pbf(\Qbb)\bsh[Y\times\Pbf(\adele)/K]\ra\Pbf'(\Qbb)\bsh[Y'\times \Pbf'(\adele)/K']$ when evaluated over $\Cbb$. It descends to a morphism of algebraic varieties defined over $E(\Pbf,Y)$ when the canonical models are taken into account.

In particular, when $(\Pbf,Y)$ is of the form $\Vbf\rtimes_\rho(\Gbf,X)$ as in \ref{kuga-datum}(1), and $K=K_\Vbf\rtimes K_\Gbf$ for $K_\Vbf\subset\Vbf(\adele)$ and $K_\Gbf\subset\Gbf(\adele)$ \cosgs, then $\Pbf(\Qbb)_+=\Vbf(\Qbb)\rtimes\Gbf(\Qbb)_+$, with $\Gbf(\Qbb)_+$ equal to the stabilizer in $\Gbf(\Qbb)$ of the connected component $X^+=\pi_*(Y^+)$, and $\Pbf(\Qbb)_+\bsh\Pbf(\adele)/K\isom\Gbf(\Qbb)_+\bsh\Gbf(\adele)/K_\Gbf$, as for the unipotent radical $\Vbf$ we have $\Vbf(\adele)=\Vbf(\Qbb)K_\Vbf$. In this case the canonical projection $\pi:M_K(\Pbf,Y)\ra M_{K_\Gbf}(\Gbf,X)$ is an abelian scheme, whose fibers are abelian varieties with $\Vbf(\Abb)/\Vbf(\Qbb)K_\Vbf$ as the underlying complex tori.

%We remark that $\Gbf(\Qbb)_+$ also equals $\Gbf(\Qbb)\cap\Gbf(\Rbb)_+$, where $\Gbf(\Rbb)_+$ is the preimage of $\Gbf^\ad(\Rbb)^+$ under the natural map $\Gbf(\Rbb)\ra\Gbf^\ad(\Rbb)$.

(3) When $K$ runs through \cosgs of $\Pbf(\adele)$, we get the total Kuga scheme as a pro-scheme $M(\Pbf,Y)=\limproj_KM_K(\Pbf,Y)$, with transition maps all defined over $E(\Pbf,Y)$. Morphisms of total Kuga schemes $f:M(\Pbf,Y)\ra M(\Pbf',Y')$ are defined in the evident way once a morphism $f:(\Pbf,Y)\ra(\Pbf',Y')$ is given; it is defined over $E(\Pbf,Y)$.

Take $g\in\Pbf(\adele)$, the evident translation $$t(g):M(\Pbf,Y)(\Cbb)\ra M(\Pbf,Y)(\Cbf)\ \ ([y,a]_K)\mapsto([y,ag]_K)$$ is a well-defined morphism of complex pro-scheme, and descends to $E(\Pbf,Y)$, called the Hecke translation by $g$.

(4) A special subvariety of $M_K(\Pbf,Y)$ is a geometrically irreducible component of the closed subvariety $\pr_K(t(g)(f(M(\Pbf_1,Y_1))))$, the latter given by the composition $$M(\Pbf_1,Y_1)\ot{f}\lra M(\Pbf,Y)\ot{t(g)}\lra M(\Pbf,Y)\ot{\pr_K}\lra M_K(\Pbf,Y)$$ where $f$ is the morphism given by the inclusion of subdatum $f:(\Pbf_1,Y_1)\mono(\Pbf,Y)$ and $\pr_K$ is the canonical projection from the projective limit. In particular, the special subvariety is a closed subvariety defined over a number field containing $E(\Pbf_1,Y_1)$.

Equivalently, the special subvariety above is given as the unique closed subvariety whose set of complex points are given by the formula  $\wp_K(Y_1^+\times gK)$, where $Y_1^+$ is some connected component of $Y_1$, and $\wp_K: Y\times\Pbf(\adele)/K\ra M_K(\Pbf,Y)(\Cbb)$ is the quotient map modulo $\Pbf(\Qbb)$, which is called uniformization of $M_K(\Pbf,Y)$. See \cite{moonen-model}  6.2 for similar description in the pure case.

\end{definition}

And we will mainly work in the setting of connected Kuga varieties(cf.\cite{chen-rigid} 1.9 (2)):
%, with certain mesaure-theoretic constructions used as a blackbox
\begin{definition}\label{connected-kuga-variety}

(1) Connected Kuga data are of the form $(\Pbf,Y;Y^+)$ with $Y^+$ a connected component of $Y$. A connected Kuga variety is a geometrically irreducible algebraic variety $M$ given by $M(\Cbb)=\Gamma\bsh Y^+$ associated to connected Kuga datum $(\Pbf,Y;Y^+)$ and congruence subgroup $\Gamma=\Pbf(\Qbb)_+\cap K$ for some \cosg $K\subset\Pbf(\adele)$. In particular it is admits a canonical model over a finite abelian extension of $E(\Pbf,Y)$, as will be explained in \ref{reciprocity}.

Morphisms between connected Kuga data and connected Kuga varieties are defined in the evident way. For example, a conencted subdatum is of the form $(\Pbf',Y';Y'^+)\subset(\Pbf,Y;Y^+)$ with $Y'^+$ a connected component of $Y'$ that is contained in $Y^+$.

In particular, from a morphism of connected Kuga data $f:(\Pbf,Y;Y^+)\ra(\Pbf',Y';Y'^+)$ and congruence subgroups $\Gamma\subset\Pbf(\Qbb)_+$ and $\Gamma'\subset\Pbf'(\Qbb)_+$ such that $f(\Gamma)\subset\Gamma'$, we get an evident map of quotient spaces $f:M(\Cbb)=\Gamma\bsh Y^+\ra M'(\Cbb)=\Gamma'\bsh Y'^+$. It descends to a morphism of algebraic varieties over some common number field $F$ containing the reflex fields $E(\Pbf,Y)$ and $E(\Pbf',Y')$.

In the sequel we simply write $M=\Gamma\bsh Y^+$ to indicate the Kuga variety defined over some number field.

(2) For $M=\Gamma\bsh Y^+$ defined by $(\Pbf,Y;Y^+)$  as above, we have the uniformization map $$\wp_\Gamma:Y^+\ra\Gamma\bsh Y^+\ \ y\mapsto\Gamma y.$$ A special subvariety is of the form $\wp_\Gamma(Y'^+)$ given by some connected subdatum $(\Pbf',Y';Y'^+)\subset(\Pbf,Y;Y^+)$. It equals the image of the morphism $\Gamma'\bsh Y'^+\ra \Gamma\bsh Y^+$ with $\Gamma'$ a congruence subgroup contained in $\Gamma\cap\Pbf'(\Qbb)_+$, and it is a geometrically irreducible subvariety of $M$ defined over some number field.

We remark that special subvarieties described in \ref{kuga-variety} (4) can be realized in the connected setting: say $M'$ is a special subvariety of $M_K(\Pbf,Y)$ given as $\wp_K(Y'^+\times gK)$ for $g\in\Pbf(\adele)$, some subdatum $(\Pbf',Y')$, and a connected component $Y'^+$ of $Y'$, then it is the same as $\wp_\Gamma(Y'^+)$ in $\Gamma\bsh Y^+$, where $Y^+$ is a connected component of $Y$ containing $Y'^+$, and $\Gamma=\Pbf(\Qbb)\cap gKg^\inv$.

%(3) If $M=\Gamma\bsh Y^+$ is a connected Kuga variety,

\end{definition}

Our later arguments also involve the description of the Galois action on the set of geometrically connected components of the canonical models.
\begin{definition-proposition}\label{canonical-model}\cite{deligne-pspm} 2.2.2 - 2.2.5

(0) Some abelian groups associated to reductive $\Qbb$-groups:

For $\Hbf$ a connected reductive $\Qbb$-group, we put $\lambda:\Hbf'\ra\Hbf^\der$ to be the simply-connected covering of $\Hbf^\der$, $\Hbf(\Qbb)^-_+$ the closure of $\Hbf(\Qbb)_+$ in $\Hbf(\adele)$, and put $\pitilde(\Hbf)$ to be the quotient $\Hbf(\Abb)/(\Hbf(\Rbb)_+\Hbf(\Qbb)^-_+\lambda(\Hbf'(\adele)))$, which is a profinite abelian group. Here it carries the natural action of translation by $\Hbf(\adele)$ given by the evident homomorphism $\iota:\Hbf(\adele)\ra\pitilde(\Hbf)$, and for $K\subset\Hbf(\adele)$ \cosg, we write $\pitilde(\Hbf)/K$ for the finite abelian group $\pitilde(\Hbf)/\iota(K)$.

In particular, for a fixed number field $F$, the $\Qbb$-torus $\Hbf=\Gbb_\mrm^F:=\Res_{F/\Qbb}\Gbb_{\mrm F}$ is used in global class field theory: the reciprocity map for the global class field theory is an isomorphism $$\rec^F:\Gal(F^\ab/F)\isom\pitilde(\Gbb^F_\mrm).$$

(1) The case of a toric Shimura datum:

Let $(\Tbf,x)$ be a toric Shimura datum, with $E=E(\Tbf,x)$, and $K\subset\Tbf(\adele)$ \cosg. The canonical model of the zero-dimensional Shimura variety $M_K(\Tbf,x)$ over $E$ is given by the action of $\Gal(\Qac/E)$ on the set $\pi_0(M_K(\Tbf,x)_\Qac)=\Tbf(\adele)/\Tbf(\Qbb)K$ by translation through the homomorphism called $\rec_\Tbf$: $$\Gal(\Qac/E)\ra\pitilde(\Gbb_\mrm^F)\ot{ r_x }\lra\pitilde(\Tbf)\lra\pitilde(\Tbf)/K$$
where $r_x$ is the homomorphism defined in \ref{reflex-field}(b).% $$\Res_{E_x/\Qbb}\Gbb_{\mrm E_x}\ot{\Res_{E_x/\Qbb}\mu_x}\lra\Res_{E_x/\Qbb}\Tbf_{E_x}\or{\Nm_{E_x/\Qbb}\lra\Rbf$$ with $\mu_x$ well-defined by definition of reflex fields.

(2) The pure case:

Let $M_K(\Gbf,X)$ be a pure Shimura variety at finite level, with $E=E(\Gbf,X)$ the reflex field. Then the set of geometrically connected components $\pi_0(M_K(\Gbf,X)_\Qac)$ is canonically identified with $\pitilde(\Gbf)/K$, and the action of $\Gal(\Qac/E)$ on it is given by the translation via the homomorphism $$\Gal(\Qac/E)\lra\Gal(E^\ab/E)\ot{\rec^E}\lra\pitilde(\Gbb^E_\mrm)\ot{\mu_X}\lra\pitilde(\Res_{E/\Qbb}\Gbf_E)\ot{\Nm_{E/\Qbb}}\lra\pitilde(\Gbf)\lra\pitilde(\Gbf)/K.$$Readers are referred to \cite{deligne-pspm} 2.4 for the construction of $\mu_X$ and $\Nm_{E/\Qbb}$ in the non-commutative case. It suffices to know that, if we put $\nu:\Gbf\ra\Tbf:=\Gbf/\Gbf^\der$ the abelianization, then $\nu:\pitilde(\Gbf)\isom\pitilde(\Tbf)$ and $\mu:\pitilde(\Gbf)/K\isom\pitilde(\Tbf)/\nu(K)$ canonically, by \cite{deligne-bourbaki} 2.4 and 2.5.

If $M_K(\Gbf,X)\ra M_{K'}(\Gbf',X')$ is a morphism as in \ref{kuga-variety}(2), the resulting map $$\pi_0(M_K(\Gbf,X)_\Qac)\ra\pi_0(M_{K'}(\Gbf',X')_\Qac)$$ is equivariant \wrt the inclusion of Galois groups $\Gal(\Qac/E(\Gbf,X))\ra\Gal(\Qac/E(\Gbf',X'))$.
 
(3) The Kuga case:

As the Andr\'e-Oort conjecture is insensitive \wrt level raising (cf.\ref{level-insensitivity} below), for Kuga vareities $M_K(\Pbf,Y)$ with $(\Pbf,Y)=\Vbf\rtimes_\rho(\Gbf,X)$ we always take the level $K$ to be of the form $K_\Vbf\rtimes K_\Gbf$, where $K_\Vbf\subset\Vbf(\adele)$ and $K_\Gbf\subset\Gbf(\adele)$ are \cosgs, and $K_\Gbf$ stabilzies $K_\Vbf$ \wrt $\rho$. Then the canonical projection $\pi:M_K(\Pbf,Y)\ra M_{K_\Gbf}(\Gbf,X)$ is an abelian scheme, and the fibers are geometrically connected. Thus $\pi_0(M_K(\Pbf,Y)_\Qac)=\pi_0(M_{K_\Gbf}(\Gbf,X))$, with the action of $\Gal(\Qac/E)$ as before, where $E=E(\Pbf,Y)=E(\Gbf,X)$. Namely the mixed case is reduced to a pure section.

The general case when $K$ is not a semi-direct product of the form above can be treated similarly. However when we worked with special subvarieties of $M_K(\Pbf,Y)$ that are given by subdata of the form $(\Pbf',Y')=\Vbf'\rtimes(v\Gbf'v^\inv, v\rtimes X')$, the set $\pi_0(M_{K\cap\Pbf'(\adele)}(\Pbf',Y'))$ becomes more complicated. It would be more convenient to assume $K=K_\Vbf\rtimes K_\Gbf$ so as to simplify some computations on the Galois conjugates of special subvarieties given by $(\Pbf',Y')$.

\end{definition-proposition}

\begin{lemma}[level insensitivity]\label{level-insensitivity}

Let $f:M=M_K(\Pbf,Y)\ra M'=M_{K'}(\Pbf,Y')$ be a mmorphism of Kuga varieties given by $f:(\Pbf,Y)\ra(\Pbf',Y')$ is a morphism of Kuga data in the sense of Pink, with $f :Y\ra Y'$ a finite covering and $f(K)\subset K'$. Then the Andr\'e-Oort conjecture holds for $M$ \ifof it holds for $M'$. The same is true when $M$ and $M'$ are replaced by connected Kuga varieties in the sense of Pink.

Let $M=\Gamma\bsh Y^+$ be a connected Kuga variety, and $\Gamma'\subset\Gamma$ a second congruence subgroup. Then the Andr\'e-Oort conjecture holds for $M'=\Gamma'\bsh Y^+$ \ifof it holds for $M$.

\end{lemma}
%Therefore in order to study the Andr\'e-Oort conjecture it suffices to restrict to the Kuga varieties defined as in \ref{kuga-variety}.

\begin{proof}
For $M$ a Kuga variety, write $\Fcal(M)$ the set of finite union of special subvarieties in $M$ (in the sense of Pink). Now that $f:Y\ra Y'$ is a finite covering, equivariant \wrt $f:\Pbf(\Rbb)\ra\Pbf'(\Rbb)$, the image and the pull-back along $f$ respect the finite union of special subvarieties, i.e. the arrows $f:\Fcal(M)\leftrightarrow\Fcal(M'):f^\inv$ are well-defined.

Assume the Andr\'e-Oort conjecture holds for $M'$, and $Z\subset M$ a closed subvariety, with $\Sigma(Z)$ the set of maximal special subvarieties in $Z$. Then for $S'$ a maximal special subvarieties in $f(Z)$, and $S$ a geometrically irreducible component of $f^\inv(S')$ contained in $Z$, we have $S\in\Sigma(Z)$ by dimension arguments.

Similarly, if the Andr\'e-Oort conjecture holds for $M$, then it holds for $M'$ using the property of $f$ and dimension arguments.\end{proof}

For $\Gbf$ a connected reductive $\Qbb$-group, with $\Cbf$ the connected center of $\Gbf$, and $\nu:\Gbf\ra\Tbf=\Gbf/\Gbf^\der$ the abelianization, the induced map $\nu:\Cbf\ra\Tbf$ is an isogeny. Take $K_\Cbf\subset\Cbf(\adele)$ \cosg, and $K_\Tbf=\nu(K_\Cbf)$, then $\nu:\pitilde(\Cbf)/K_\Cbf\ra\pitilde(\Tbf)/K_\Tbf$ is a homomorphism of finite abelian groups; its kernel and cokernel are of cardinality bounded by a constant which only depends on $\Gbf$. See e.g. \cite{ullmo-yafaev} 2.8 for related proofs.
\bigskip

We recall the notions of $\Cbf$-special subvarieties and rigid Kuga varieties (cf.\cite{chen-rigid} 2.1):

\begin{definition}\label{c-special-rigid}
Let $(\Pbf,Y)=\Vbf\rtimes_\rho(\Gbf,X)$ be a Kuga datum, and fix $\Cbf$ a connected $\Qbb$-torus in $\Gbf$. Write $\pi:(\Pbf,Y)\ra(\Gbf,X)$ for the canonical quotient (modulo $\Vbf$)

(1) A Kuga subdatum $(\Pbf',Y')$ of $(\Pbf,Y)$ is said to be $\Cbf$-special, if $\Cbf$ equals the connected center of the image $\pi(\Pbf')$ (with $(\pi(\Pbf'),\pi_*(Y'))$ a pure subdatum of $(\Gbf,X)$). The case of connected Kuga data is similar.

 Let $M=\Gamma\bsh Y^+$ be a connected Kuga variety. A special subvariety $M'$ of $M$ is said to be $\Cbf$-special if it is given as $M'=\wp_\Gamma(Y'^+)$ for some connected $\Cbf$-special subdatum $(\Pbf',Y';Y'+)\subset(\Pbf,Y;Y^+)$.

And by the reduction in \cite{chen-rigid} 2.8, the study of the Andr\'e-Oort conjecture for $\Cbf$-special subvarieties in a given connected Kuga variety $M$ is reduced to the case where $M$ is a connected Kuga variety given by $(\Pbf,Y;Y^+)=\Vbf\rtimes_\rho(\Gbf,X;X^+)$ with $\Cbf$ the connected center of $\Gbf$.

(2) $(\Pbf,Y)$ is said to be rigid (\wrt $\rho$), if $\Vbf^{\Gbf^\der}=0$, namely the action of $\Gbf^\der$ on $\Vbf$ through $\rho$ does not admit any trivial subrepresentation of dimension $>0$.

A subdatum $(\Pbf',Y')=\Vbf'\rtimes(v\Gbf'v^\inv, v\rtimes X')$ of $(\Pbf,Y)$ is rigid if $\Vbf\rtimes_\rho(\Gbf',X')$ is rigid, i.e. $\Vbf^{\Gbf'^\der}=0$. Note that this condition is stronger than $\Vbf'^{\Gbf'^\der}=0$. And if $(\Pbf,Y)$ admits a rigid subdatum, then $(\Pbf,Y)$ is rigid itself.

Similarly, we have the notions of rigid connected Kuga (sub)data and of rigid special subvarieties: for $M=\Gamma\bsh Y^+$ a connected Kuga variety given by $(\Pbf,Y;Y^+)$, a special subvariety of $M$ is rigid if it is given by some rigid connected subdatum.

\end{definition}

We end this section with the following examples of Kuga varieties:

\begin{example}
  \label{example}(cf. \cite{pink-combination} 2.7, 2.8. 2.10)
(1) Let $(\Vbf,\psi)$ be a symplectic $\Qbb$-vector space of dimension $2g$, $\GSp(\Vbf)$ the $\Qbb$-group of symplectic similitude of $(\Vbf,\psi)$, and $\Hscr(\Vbf)$ the set of homomorphisms of $\Rbb$-groups $h:\Sbb\ra\GSp(\Vbf)_\Rbb$ that induces a complex structure on $\Vbf$ via $\GSp(\Vbf)\subset\GLbf(\Vbf)$ which is $\psi$-polarizable, i.e. $(x,y)\mapsto\psi(x,h(\ibf)y)$ is symmetric and definite. Then $\Hscr(\Vbf)$ is isomorphic to the Siegel double upper half space, and $(\GSp(\Vbf),\Hscr(\Vbf))$ is a Shimura datum. The reflex field is $\Qbb$.

$(\Vbf,\psi)$ admits an integral basis $(e_{\pm i})_{1\leq i\leq g}$, i.e. a $\Qbb$-basis $(e_{\pm i})$ such that $\psi(e_i,e_j)=0$ for $i+j\neq 0$, and $\psi(e_i,e_{-i})=1$ for $1\leq i\leq g$. Then $\Lbf=\oplus\Zbb e_i$ is a $\Zbb$-lattice for $\Vbf$, and by taking the \cosg $K(n)=\Ker(\GSp(\Zbhat\otimes_\Zbb\Lbf)\ra\GSp(\Lbf/n\Lbf))$ of $\GSp(\Vbf)(\adele)$ we get $M_{K(n)}(\GSp(\Vbf),\Hscr(\Vbf))$ the moduli $\Qbb$-scheme of principally polarized abelian scheme with level-n structure.

Quotient by $\Sp(\Vbf)$ defines a morphism of pure Shimura data $$(\GSp(\Vbf),\Hscr(\Vbf))\ra(\Gbb_\mrm,\cdot)$$ and a morphism of Shimura varieties $$M_{K(n)}(\GSp(\Vbf),\Hscr(\Vbf))\ra M_{K_{\Gbb_\mrm}(n)}(\Gbb_\mrm,\cdot)$$ with geometrically connected fibers, where the target is $(\Zbb/n)^\times$, on which $\Gal(\Qac/\Qbb)$ acts through the cyclotomic Galois group $(\Zbb/n)^\times\isom\Gal(\Qbb(\mu_n)/\Qbb)$. When $K(n)$ shrinks to the trivial group, we get $\Zbhat^\times\isom\Gal(\Qbb^\ab/\Qbb)$ as the set of geometrically connected components of $M(\GSp(\Vbf),\Hscr(\Vbf))$.

 The standard representation $\rho:\GSp(\Vbf)\ra\GLbf_\Qbb(\Vbf)$ gives rise to the Kuga datum $(\Pcal(\Vbf),\Ycal(\Vbf)):=\Vbf\rtimes_\rho(\GSp(\Vbf),\Hscr(\Vbf))$. And by taking the \cosg $K_\Vbf=\Zbhat\otimes_\Zbb\Lbf$ of $\Vbf(\adele)$ and $K=K_\Vbf\rtimes K(n)$  the canonical projection $$\pi:M_K(\Pcal(\Vbf),\Ycal(\Vbf))\ra M_{K(n)}(\GSp(\Vbf),\Hscr(\Vbf))$$ defines an abelian scheme, which is the universal family of abelian vareities corresponding to the moduli $\Qbb$-scheme $M_{K(n)}(\GSp(\Vbf),\Hscr(\Vbf))$.

(2) Keep the notations as in (1), and consider a toric subdatum $(\Tbf,x)\subset(\GSp(\Vbf),\Hscr(\Vbf))$. $x:\Sbb\ra\GSp(\Vbf)_\Rbb$ is then a homomorphism with image in $\Tbf_\Rbb$, and $\Tbf$ is the smallest $\Qbb$-subgroup having this property. The Kuga datum $(\Pbf,Y):=\Vbf\rtimes_\rho(\Tbf,x)$ gives rise to connected Kuga varieties of the form $\Gamma_\Vbf\bsh\Vbf(\Rbb)$ which is an abelian variety with Mumford-Tate group equal to $\Tbf$: say we take $K_\Tbf=\Tbf(\adele)\cap K(n)$, and $\Gamma_\Vbf=\Vbf(\Qbb)\cap K_\Vbf$, then $\Gamma_\Vbf\bsh\Vbf(\Rbb)$ is realized as a connected component of $M_{K_\Vbf\rtimes K_\Tbf}(\Pbf,Y)$, which is the pull-back of the abelian scheme $\pi:M_{K_\Vbf\rtimes K(n)}(\Pcal(\Vbf),\Ycal(\Vbf))\ra M_{K(n)}(\GSp(\Vbf),\Hscr(\Vbf))$ along the subscheme $M_{K_\Tbf}(\Tbf,x)\ra M_{K(n)}(\GSp(\Vbf),\Hscr(\Vbf))$ which is a zero-dimensional subscheme defined over $E(\Tbf,x)$.

\end{example}

\section{Decomposition of Kuga varieties}

In this section we decompose a connected Kuga variety, up to changing the level, into a product of three Kuga varieties $A\times M^\rig\times S$, where $A$ is a CM abelian variety obtained as in \ref{example} (2), $S$ a connected pure Shimura variety, and $M^\rig$ a connected Kuga variety given by some rigid Kuga datum $(\Pbf,Y;Y^+)=\Vbf\rtimes_\rho(\Gbf,X;X^+)$. The decomposition is compatible with the structure of canonical models.

\begin{lemma}\label{split-tori} Let $f:(\Gbf,X)\ra(\Gbf',X')$ be a morphism of pure Shimura data, such that $f(\Gbf)$ is normal in $\Gbf'$, and that the neutral components of $\Ker(f)$ and of $\Cok(f)$ are split $\Qbb$-tori. Then $E(\Gbf,X)=E(\Gbf',X')$.

\end{lemma}
\begin{proof}

Take $x\in X$, such that $x(\Sbb)\subset\Tbf_\Rbb$ for some maximal $\Qbb$-torus $\Tbf\subset\Gbf$. Then $x'=f(x)$ sends $\Sbb$ into $f(\Tbf)_\Rbb$. Extend $f(\Tbf)$ to a maximal $\Qbb$-torus $\Tbf'$. Then the restriction $f:\Tbf\ra\Tbf'$ is of $\Qbb$-split kernel and cokernel. And by our assumption, $\Ker(f:\Tbf\ra\Tbf')$ is central in $\Gbf$, and $\Cok(f:\Tbf\ra\Tbf')$ is central in $\Gbf'$.

Consequently $f$ extends to normalizers $\Nbf_\Gbf(\Tbf)\ra\Nbf_{\Gbf'}(\Tbf')$, with $W(\Gbf_\Cbb,\Tbf_\Cbb)\isom W(\Gbf'_\Cbb,\Tbf'_\Cbb)$ induced. Moreover, the induced map $\Xbf^\vee(f):f:\Xbf^\vee_\Tbf(\Cbb)\ra\Xbf_{\Tbf'}^\vee(\Cbb)$ is equivariant \wrt the actions of $\Aut(\Cbb/\Qbb)$ and $W(\Gbf_\Cbb,\Tbf_\Cbb)\ra W(\Gbf'_\Cbb,\Tbf'_\Cbb)$, and the actions of these two groups are trivial on the kernel and the cokernel of $\Xbf^\vee(f)$.

Write $E=E(\Gbf,X)$ and $E'=E(\Gbf',X')$, with $\mu=\mu_x:$ and $\mu'=\mu_{x'}=f_*\mu_x$. By \ref{reflex-field}, $\Aut(\Cbb/E)$ resp. $\Aut(\Cbb/E')$ is the stabilizer of $[\mu]$ resp. of $[\mu']$ in $\Xbf^\vee_\Tbf(\Cbb)/W(\Gbf\Cbb,\Tbf_\Cbb)$ resp. $\Xbf^\vee_{\Tbf'}(\Cbb)/W(\Gbf'_\Cbb,\Tbf'_\Cbb)$. Now that  the diferrence between $\Xbf^\vee_\Tbf(\Cbb)$ and $\Xbf^\vee_{\Tbf'}(\Cbb)$ are given by split $\Qbb$-tori, and the Weyl groups are identified, we get the equality $\Aut(\Cbb/E)=\Aut(\Cbb/E')$.\end{proof}

\begin{lemma}\label{pure-kuga-decomposition} Let $(\Pbf,Y)$ be a Kuga datum of the form $\Vbf\rtimes_\rho(\Gbf,X)$ as in \ref{kuga-datum}(1). Write $\Vbf=\oplus\Vbf_i$ for the decomposition of $\Vbf$ into simple $\Qbb$-representations of $\Gbf$. Then

(1) For each $i$, $\Gbf$ stabilizes a symplectic form $\psi_i$ on $\Vbf$ up to scalar, and $\rho$ gives rise to a morphism of pure Shimura data $\rho:(\Gbf,X)\ra\prod_i(\GSp(\Vbf_i),\Hscr(\Vbf_i))$. We write $$\rho:(\Pbf,Y)\ra\prod_i(\Pcal(\Vbf_i),\Ycal(\Vbf_i))=\prod_i\Vbf_i\rtimes(\GSp(\Vbf_i),\Hscr(\Vbf_i))$$ for the extension by $\Vbf$, which fits into a cartesian diagram  $$\xymatrix{ {(\Pbf,Y)} \ar[r]_\rho \ar[d]_\pi & {\prod_i(\Pcal(\Vbf_i),\Ycal(\Vbf_i))} \ar[d]_\pi\\ {(\Gbf,X)} \ar[r]_\rho & {\prod_i(\GSp(\Vbf_i),\Hscr(\Vbf_i))}}$$

(2) There exists an epimorphism of pure Shimura data $\rho':(\Gbf,X)\ra(\Gbf',X')$ such that the resulting morphism $(\Gbf,X)\ot{(\rho',\rho)}\lra(\Gbf',X')\times\prod_i(\GSp(\Vbf_i),\Hscr(\Vbf_i))$ is   of finite kernel. We also write $(\rho',\rho)$ for the morphism $(\Pbf,Y)\ra(\Gbf',X')\times\prod_i(\Pcal(\Vbf_i),\Ycal(\Vbf_i))$, whose kernel is finite.

In particular, $E(\Pbf,Y)=E(\Gbf,X)$ is the composition of $E(\Gbf',X')$ with $E(\rho(\Gbf),\rho_*(X))$.

\end{lemma}

\begin{proof}
(1) See e.g. \cite{milne-moduli} 2.1, 6.3 and 10.3. %Recall \cite{deligne-pspm} ?.? that for any algebraic representation $\Gbf\ra\GLbf(\Wbf)$, the constant local system $\Wbf\times X$ underlies a variation of polarizable rational Hodge structures. Since the local system is constant, take fibers over any base point $x\in X$ gives rise to a polarized Hodge structure $(\Wbf,\rho\circ x)$. In our case $\Wbf$ is just $(\Vbf,\rho)$, and we conclude that $\Gbf$ respects a polarization $\Vbf\otimes_\Qbb\Vbf\ra\Qbb(-1)$ up to scalars, which is given by a symplectic form as the weight of $\Vbf$ is odd. See also \cite{milne-introduction} ?.?.

%Thus we get $(\Gbf,X)\ra(\GSp(\Vbf),\Hscr(\Vbf))$ for some symplectic form $\psi$ on $\Vbf$. And by pull-back we get the cartesian diagram.

(2) Write $\Hbf$ for the neutral component of $\Ker(\rho:\Gbf\ra\prod_i\GSp(\Vbf_i))$.  It suffices to find a normal $\Qbb$-subgroup $\Nbf$ of $\Gbf$ such that map $\Nbf\times\Hbf$  induced by inclusions into $\Gbf$ is surjective with finite kernel.

At the level of derived groups, we have $\Hbf^\der$ as a normal $\Qbb$-subgroup of $\Gbf^\der$, both being connected and semi-simple. By taking the decomposition of $\Gbf^\der$ into the almost direct product of simple $\Qbb$-factors, we find a connected semi-simple normal $\Qbb$-subgroup $\Nbf'$ such that $\Nbf'\Hbf^\der=\Gbf^\der$ and that $\Nbf'\cap\Hbf^\der$ is finite.

It remains to find a $\Qbb$-torus serving as the connected center of $\Nbf$. Write $\Cbf$ resp. $\Cbf_\Hbf$ for the connected center of $\Gbf$ resp. of $\Hbf$. Then there exists a $\Qbb$-subtorus $\Cbf'$ of $\Cbf$ such that $\Cbf'\cap\Cbf_\Hbf$ is finite and $\Cbf=\Cbf'\Cbf_\Hbf$. It is clear that $\Cbf'$ centralizes $\Nbf'$, and it suffices to take $\Nbf=\Cbf'\Nbf'$.\end{proof}

\begin{lemma}\label{faithful-rigid-decomposition}%Let $(\Vbf,\psi)$ be a symplectic $\Qbb$-vector space, $(\GSp(\Vbf),\Hscr(\Vbf))$ the associated Siegel-Shimura datum, and $(\Gbf,X)\subset(\GSp(\Vbf),\Hscr(\Vbf))$ a pure subdatum. Then we have a decomposition $(\Vbf,\rho)=(\Ubf,\sigma)\oplus(\Wbf,\tau)$ into subrepresentations of $\Gbf$, such that

Let $(\Pbf,Y)=\Vbf\rtimes_\rho(\Gbf,X)$ be a Kuga datum, such that $\rho:\Gbf\ra\GLbf_\Qbb(\Vbf)$ is faithful. Decompose $\Vbf=\oplus_{i\in I\coprod J}\Vbf_i$ into simple $\Qbb$-subrepresentations of $\Gbf$, such that $\Gbf^\der$ acts on $\Vbf_i$ trivially for $i\in I$, and $\Vbf_j^{\Gbf^\der}=0$ for $j\in J$. Write $\Ubf=\oplus_I\Vbf_i$ and $\Wbf=\oplus_J\Vbf_j$. Then by choosing symplectic forms $\psi_i$ on $\Vbf_i$, we get morphisms of pure Shimura data $$\alpha:(\Gbf,X)\ra(\Gbf_I,X_I):=\prod_I(\GSp(\Vbf_i),\Hscr(\Vbf_i))$$ $$\beta:(\Gbf,X)\ra(\Gbf_J,X_J):=\prod_J(\GSp(\Vbf_j),\Hscr(\Vbf_j)).$$ Write $(\Gbf_\alpha,X_\alpha)$ resp. $(\Gbf_\beta,X_\beta)$ for the image of $\alpha$ resp. $\beta$. Then

(1) $\Gbf^\der\subset\Ker(\alpha:\Gbf\ra\Gbf_I)$, and the image of  $\alpha$ is a toric subdatum of $(\Gbf_I,X_I)$, while $\Wbf^{\Gbf^\der_\beta}=0$.

(2)   $(\alpha,\beta):(\Gbf,X)\ra(\Gbf_\alpha,X_\alpha)\times(\Gbf_\beta,X_\beta)$ satisfies the conditions in \ref{split-tori}, and the same is true for $$(\Ubf\oplus\Wbf)\rtimes_{(\alpha,\beta)}(\Gbf,X)\ra(\Ubf\rtimes(\Gbf_\alpha,X_\alpha))\times(\Wbf\rtimes(\Gbf_\beta,X_\beta)).$$ In particular, $E(\Gbf,X)=E(\Gbf_\alpha,X_\alpha)E(\Gbf_\beta,X_\beta)$.
\end{lemma}

\begin{proof}

(1) This is clear by the definition of $(\Ubf,\alpha)$ and $(\Wbf,\beta)$.

(2) It is clear that $(\alpha,\beta)$ is injective as $\rho:\Gbf\ra\GLbf_\Qbb(\Vbf)$ is already injective. Now that $\Gbf_\alpha$ is a $\Qbb$-torus, the image of $\Gbf^\der$ falls into $\Gbf_\beta$, which differs from $\Gbf_\beta$ by a central $\Qbb$-torus. It follows that the cokernel of $(\alpha,\beta)$ is a $\Qbb$-torus.

Now that the action of $\Gbf_\beta$ on $\Wbf$ is faithful and rigid, by \ref{rigidity-characterization} we see that the connected center of $\Gbf_\beta$ is a split $\Qbb$-tori. As a consequence, the compact part of the connected center of $\Gbf$ is killed by $\beta$, and is thus sent into $\Gbf_\alpha$ injectively by $\alpha$. Since $\alpha(\Gbf)=\Gbf_\alpha$, the cokernel of $\alpha$ is also split. Hence $(\alpha,\beta)$ satisfies \ref{split-tori}.

The case of Kuga data is similar: it suffices to see that the morphism $$(\Ubf\oplus\Wbf)\rtimes_{(\alpha,\beta)}(\Gbf,X)\ra(\Ubf\rtimes(\Gbf_\alpha,X_\alpha))\times(\Wbf\rtimes(\Gbf_\beta,X_\beta))$$ is injective and is isomorphic on the unipotent radical $\Ubf\otimes\Wbf$, and thus the cokernel is the same as the pure case.\end{proof}

From the discussion above we deduce the following

\begin{corollary}\label{triple-product-decomposition} %Let $M=\Gamma\bsh Y^+$ be a connected Kuga variety defined by $(\Pbf,Y)=\Vbf\rtimes_\rho(\Gbf,X)$, $\Gamma\subset\Pbf(\Qbb)_+$ a congruence subgroup. Then

Let $(\Pbf,Y)=\Vbf\rtimes_\rho(\Gbf,X)$ be a Kuga datum. Then there exists a morphism of Kuga data $f:(\Pbf,Y)\ra(\Pbf^\cm,Y^\cm)\times(\Pbf^\rig,Y^\rig)\times(\Gbf',X')$, where $(\Pbf^\cm,Y^\cm)=\Ubf\rtimes(\Tbf,x)$ is a Kuga datum with $(\Tbf,x)$ toric, $(\Pbf^\rig,Y^\rig)=\Wbf\rtimes(\Gbf^\rig,X^\rig)$ a rigid Kuga datum, and $(\Gbf',X')$ a pure Shimura datum, with $\Tbf\ra\GLbf_\Qbb(\Ubf)$ and $\Gbf^\rig\ra\GLbf_\Qbb(\Wbf)$ both faithful, such that:

(1) $f:\Gbf\ra\Tbf\times\Gbf^\rig\times \Gbf'$ and  $f:\Pbf\ra\Pbf^\cm\times\Pbf^\rig\times\Gbf'$ satisfy the conditions in \ref{split-tori} with finite kernel, and $E(\Gbf,X)$ equals the composition of $E(\Tbf,x)$, $E(\Gbf^\rig,X^\rig)$ and $E(\Gbf',X')$. %and the neutral component of $\Cok(f)$ is finite; we also have $E=E'E''$, with $E$ resp. $E'$ resp. $E''$ the reflex field of $(\Pbf,Y)$ resp. of $(\Pbf',Y')$ resp. of $(\Pbf'',Y'')$;

(2) Take congruence subgroups $\Gamma^\cm=\Gamma_\Ubf\rtimes\Gamma_\Tbf\subset\Pbf^\cm(\Qbb)^+$, $\Gamma^\rig=\Gamma_\Wbf\rtimes\Gamma_{\Gbf^\rig}\subset\Pbf^\rig(\Qbb)^+$,  $\Gamma'\subset\Gbf'(\Qbb)^+$, and $\Gamma=f^\inv(\Gamma^\cm\times\Gamma^\rig\times\Gamma')$, then the induced map of connected Kuga varieties $$f:\Gamma\bsh Y^+\ra \Gamma^\cm\bsh Y^\cm\times \Gamma^\rig\bsh Y^{\rig,+}\times\Gamma'\bsh X'^+$$ is an isomorphism, with $\Gamma^\cm\bsh Y^\cm$ an CM abelian variety, $\Gamma^\rig\bsh Y^{\rig,+}$ a connected rigid Kuga variety, and $\Gamma'\bsh X'^+$ a connected pure Shimura variety. The isomorphism respects the structures of canonical models of some number field $F$ containing the reflex fields of $(\Pbf,Y)$.

\end{corollary}

Hence up to changing the level, a connected Kuga variety admits a decomposition into a triple product of the form above. We call $(\Pbf^\cm,Y^\cm)$ resp. $(\Pbf^\rig,Y^\rig)$ resp. $(\Gbf',X')$ the CM factor resp. the rigid factor resp. the pure factor of $(\Pbf,Y)$.%, and $M^\cm$ resp. $M^\rig$ resp. $S'$ the CM factor resp. the rigid factor resp. the pure factor of $M$ up to isogeny.

\begin{proof}

(1) The affirmation for Kuga data is just a combination of \ref{pure-kuga-decomposition} and \ref{faithful-rigid-decomposition}.

(2) Since $f:\Pbf\ra\Pbf^\cm\times\Pbf^\rig\times\Gbf'$ is of finite kernel, $\Gamma=f^\inv(\Gamma^\cm\times\Gamma^\rig\times\Gamma')$ is a congruence subgroup, and we have $\Gamma=\Gamma_\Vbf\rtimes\Gamma_\Gbf$, where $\Gamma_\Vbf=\Gamma_\Ubf\oplus\Gamma_\Wbf$, and $\Gamma_\Gbf=f^\inv(\Gamma_\Tbf\times\Gamma_{\Gbf^\rig}\times\Gamma')$. By \ref{split-tori}, one can find a central split $\Qbb$-torus $\Hbf$ in $\Tbf\times\Gbf^\rig\times\Gbf'$ such that $f(\Gbf)\Hbf=\Tbf\times\Gbf^\rig\times\Gbf'$. Therefore $f(\Gamma_\Gbf)$ only differs from $\Gamma_\Tbf\times\Gamma_{\Gbf^\rig}\times\Gamma'$ by a central part stabilizing $\Gamma_\Ubf\oplus\Gamma_\Wbf=\Gamma_\Vbf$, hence the required isomorphism.

As for the compatibility with canonical models, it suffices to take $F$ as the composition of the fields of definition of $\Gamma^\cm\bsh Y^\cm$, of $\Gamma^\rig\bsh Y^{\rig,+}$ and of $\Gamma'\bsh X'^+$, which are finite abelian extensions of the respective reflex fields.
\end{proof}

\begin{lemma}\label{subdatum-compatibility}Let $(\Pbf,Y)=\Vbf\rtimes(\Gbf,X)=(\Pbf^\cm,Y^\cm)\times(\Pbf^\rig,Y^\rig)\times(\Gbf',X')$ be a decomposition of the type in \ref{kuga-variety-decomposition}, with $(\Pbf^\cm,Y^\cm)=\Ubf\rtimes(\Tbf,x)$, $(\Pbf^\rig,Y^\rig)=\Wbf\rtimes(\Gbf^\rig,X^\rig)$, $\Cbf^\rig$ the connected center of $\Gbf^\rig$, and $\Cbf'$ the connected center of $\Gbf'$. Write $\Cbf=\Tbf\times\Cbf^\rig\times \Cbf'$ for the connected center of $\Gbf$. Then a $\Cbf$-special subdatum of $(\Pbf,Y)$ is of the form $(\Pbf'^\cm,Y'^\cm)\times(\Pbf'^\rig,Y'^\rig)\times(\Gbf'',X'')$, with $(\Pbf'^\cm,Y'^\cm)\subset(\Pbf^\cm,Y^\cm)$ a $\Tbf$-special subdatum, $(\Pbf'^\rig,Y'^\rig)\subset(\Pbf^\rig,Y^\rig)$ a $\Cbf^\rig$-subdatum, and $(\Gbf'',X'')\subset(\Gbf,X)$ a $\Cbf'$-special subdatum.

\end{lemma}
\begin{proof}
It suffices to check that $\Cbf$-special subdata of $(\Tbf,x)\times(\Gbf^\rig,X^\rig)\times(\Gbf',X')$ are of the product form required in the claim.

Write $\pr^\cm$ resp. $\pr^\rig$ resp. $\pr'$ the projection from $(\Gbf,X)$ onto the respective pure factors. Then for a $\Cbf$-special subdatum $(\Hbf,X_\Hbf)$ of $(\Gbf,X)$, its image under $\pr^\cm$ resp. $\pr^\rig$ resp. $\pr'$ equals $(\Tbf,x)$ resp. is $\Cbf^\rig$-special in $(\Gbf^\rig,X^\rig)$ resp. is $\Cbf'$-special in $(\Gbf',X')$. Write $(\Hbf_1,X_1)$ for $\pr^\rig(\Hbf,X_\Hbf)$ and $(\Hbf_2,X_2)$ for $\pr'(\Hbf,X_\Hbf)$. It remains to show that the projection of $(\Hbf,X_\Hbf)$ in $(\Gbf^\rig,X^\rig)\times(\Gbf',X')$ is of the form $(\Hbf_1,X_1)\times(\Hbf_2,X_2)$. %Since $(\Tbf,x)$ plays no role here, we assume for simplicity that the factor $(\Tbf,x)$ does not appear% $\Hbf^\der=1\times\Hbf^{\rig,\der}\times\Hbf'^\der$ is also of the product form.

Assume for the moment that  $(\Pbf,Y)$ is rigid, which excludes the factor $(\Tbf,x)$. If $(\Hbf,X_\Hbf)$ is not of the product form, then using arguments in Goursat's lemma  (cf.\cite{lang-algebra} Chap.I, ex.5), one sees that $(\Hbf,X_\Hbf)\subset(\Gbf,X)$ is given as the graph of some epimorphism $(\Hbf_1,X_1)\ra(\Hbf_2,X_2)\subset(\Gbf',X')$. It then follows that the connected center of $\Hbf$ is also the graph of $\Cbf^\rig\ra\Cbf'$, which is a proper $\Qbb$-subtorus of $\Cbf=\Cbf^\rig\times\Cbf'$, contradicting the assumption that $(\Hbf,X_\Hbf)$ is $\Cbf$-special.

The case when $(\Tbf,x)$ appears is similar. Thus $(\Hbf,X_\Hbf)$ is of the product form as is required.\end{proof}

\section{Galois conjugates of Kuga varieties}

In this section we review some functorial properties of canonical models, and state the main theorem and the main corollary. A more detailed sketch is found in the appendix.

\begin{definition-proposition}\label{conjugate}[cf.\ref{functoriality}]
Let $(\Pbf,Y)$ be a Kuga datum, $E=E(\Pbf,Y)$ its reflex field, and $\tau\in\Aut(\Cbb/\Qbb)$ an arbitrary element. Then there exists a Kuga datum $(\Pbf^\tau,Y^\tau)$ with reflex field $\tau(E)$, an isomorphism of topological groups $\psi^\tau:\Pbf(\adele)\ra\Pbf^\tau(\adele)$, and an isomorphism of proschemes over $\tau(E)$ $$\phi^\tau:\tau(M(\Pbf,Y))\ra M(\Pbf^\tau,Y^\tau)$$ equivariant \wrt $\psi^\tau:\Pbf(\adele)\ra\Pbf^\tau(\adele)$ (by Hecke translations from the right). The isomorphism is functorial in the following sense:

(1) When $\tau\in\Aut(\Cbb/E)$, $(\Pbf^\tau,Y^\tau)$ is isomorphic to $(\Pbf,Y)$ as Kuga data, and $\phi^\tau$ is an isomorphism of proschemes over $E$.

(2) Take $K\subset\Pbf(\adele)$ \cosg, and $K^\tau=\psi^\tau(K)\subset\Pbf^\tau(\adele)$. Then $\phi^\tau$ induces an isomorphism at finite level $$\phi^\tau:\tau(M_K(\Pbf))\ra M_{K^\tau}(\Pbf^\tau,Y^\tau)$$ over $\tau(E)$.

(3) If $f:(\Pbf,Y)\ra(\Pbf',Y')$ is a morphism of Kuga data, then $\tau$ transfers it into a morphism $f^\tau:(\Pbf^\tau,Y^\tau)\ra(\Pbf'^\tau,Y'^\tau)$, together with a commutative diagram of proschemes over $\tau(E(\Pbf,Y))$ $$\xymatrix{ \tau(M(\Pbf,Y)) \ar[d]^{\tau(f)} \ar[r]^{\phi^\tau} & M(\Pbf^\tau,Y^\tau) \ar[d]^{f^\tau} \\ \tau(M(\Pbf',Y')) \ar[r]^{\phi^\tau} & M(\Pbf'^\tau,Y'^\tau)}$$ which is equivariant with to the Hecke translations in the commutative diagram below

$$\xymatrix{ \Pbf(\adele) \ar[d]^f \ar[r]^{\psi^\tau} & \Pbf^\tau(\adele) \ar[d]^{f^\tau} \\ \Pbf'(\adele) \ar[r]^{\psi^\tau} & \Pbf'^\tau(\adele)}$$ with isomorphic horizontal arrows. In particular it also induces commuttive diagrams at finite levels in the sense of (2) above.

(4) When $(\Pbf,Y)=(\Tbf,x)$ is toric, $(\Tbf^\tau,x^\tau)$ is isomorphic to $(\Tbf,x_\tau)$, where $x_\tau:\Sbb\ra\Tbf_\Rbb$ is the unique homomorphism of $\Rbb$-tori such that $\mu_{x_\tau}=\mu_x\circ\tau$.

Moreover the construction of $(\Pbf^\tau,Y^\tau)$, $\phi^\tau$, $\psi^\tau$, etc. are uniquely determined by the toric case (through functoriality).

It should be point out that the constructions in \cite{milne-shih-conjugate} actually starts with $\tau\in\Aut(\Cbb/\Qbb)$ plus a toric subdatum $(\Tbf,x)\subset(\Pbf,Y)$, and it is then proved that the conjugates with $\tau$ fixed and $(\Tbf,x)$ varied are isomorphic canonically, and for simplicity we omit the complete description.

\end{definition-proposition}

Relevant constructions and proofs for the general case of mixed Shimura varieties are sketched in the appendix.

\begin{corollary}\label{finite-conjugate} Let $(\Gbf,X)$ be a pure Shimura datum and $(\Gbf',X')$ a pure subdatum with $\Cbf$ the connected center of $\Gbf'$. Fix an isomorphism of $(\Gbf,X)$ and $(\Gbf^\tau,X^\tau)$, and identify $(\Gbf'^\tau,X'^\tau)$ as a second subdatum of $(\Gbf,X)$. Write $\Cbf^\tau$ for the connected center of $\Gbf'^\tau$.

%(1) when $\tau$ runs through $\Aut(\Cbb/E)$, only finitely many $\Qbb$-tori in $\Gbf$ appear in the form $\Cbf^\tau$;

Let $S=\Gamma\bsh X^+$ be a connected pure Shimura datum associated to $(\Gbf,X;X^+)$, with its canonical model defined over a number field $F$ containing $E$, then there exists a number field $F$ containing $E$ over which $S$ admits a canonical model, such that for any $\tau\in\Aut(\Cbb/F)$ and any $\Cbf$-special subvariety $S'$ of $S$, the conjugate $\tau(S')$ is still a $\Cbf$-special subvariety of $S$.% and $S'=\wp_\Gamma(X'^+)$ a $\Cbf$-special subvariety associated to some pure $\Cbf$-special subdatum $(\Gbf',X';X'^+)$, then the number of geometrically irreducible components of the $\Gal(\Qac/F)$-orbit of $S'$ in $S$ is bounded from above by a constant independent of $S'$.
\end{corollary}

Recall that in \cite{chen-rigid} we have proved:

\begin{theorem}[cf.\cite{chen-rigid} 2.6]\label{rigid-compactness}Let $M$ be a connected Kuga variety defined by some Kuga datum $(\Pbf,Y)=\Vbf\rtimes_\rho(\Gbf,X)$, and $\Cbf$ a fixed $\Qbb$-torus of $\Gbf$. Write $\Pscr'(M)$ for the set of canonical measures on $M(\Cbb)$ associated to rigid $\Cbf$-special subvarieties of $M$. Then $\Pscr'(M)$ is compact for the weak-* topology, and admits ''support convergence'' in the sense that if $(\mu_n)_n$ is a sequence in $\Pscr'(M)$ that converges to some $\mu$, then $\Supp\mu_n\subset\Supp\mu$ for $n$ large enough, and thus $\bigcup_{n\gg 0}\Supp\mu_n$ is dense in $\Supp\mu$ for the archimedean topology.

\end{theorem}

Applying the results above to rigid Kuga varieties we get:

\begin{corollary}\label{rigid-equidistribution} Let $M$ be a connected Kuga variety defined by a connected Kuga datum  $(\Pbf,Y;Y^+)=\Vbf\rtimes_\rho(\Gbf,X;X^+)$ which is $\rho$-rigid, with $F\subset \Cbb$ a field of definition over which $M$ admits a canonical model. Write $\Cbf$ for the connected center of $\Gbf$, and let $(M_n)_n$ be a sequnence of rigid $\Cbf$-special subvarieties in $M$, $\Cbf$-strict in the sense of \ref{statement-main-theorem}. Then the $\Gal(\Qac/F)$-orbits of the $M_n$'s are equidistributed in $M$ \wrt the canonical measure on $M(\Cbb)$, in the sense of \ref{statement-main-theorem}.
\end{corollary}

\begin{proof} 
 
Write $\mu_n$ for the canonical measure associated to $M_n$. We first show that $\mu_n$ converges to $\mu$ the canonical measure associated to $M$ for the weak-* topology. Since $(\mu_n)_n$ does admits convergent subsequences by the weak-* compactness of $\Pscr'(M)$, it suffices to show that all convergent subsequences $(\mu_n')$ of $(\mu_n)$ have the same limit: if there are two subsequences $(\mu_n')_n$ and $(\mu''_n)_n$ with distinct limits $\mu'$ and $\mu''$, then we may assume that $\mu'\neq\mu$, hence $\Supp\mu'\subset\Supp\mu=M$ is a proper $\Cbf$-special subvariety. The convergence implies that $\Supp\mu'_n\subset\Supp\mu'$ for $n$ large, which contradicts the assumption of $\Cbf$-strictness. Hence $(\mu_n)_n$ converges to $\mu$.

We then pass to the convergence of average over the Galois orbits. Since $F$ contains the reflex field of $(\Gbf^\ab,x^\ab)$, with $\Gbf^\ab$ isogeneous to $\Cbf$, the corollary \ref{functoriality-corollary}   implies that each $\Gal(\Qac/F)$-conjugates of $M_n$ is again $\Cbf$-special. Write $\Oscr_n$ for the $\Gal(\Qac/F)$-conjugates of $M_n$, of cardinality $d_n$: $\Oscr_n=\{M_n=M_{n,1},\cdots,M_{n,d_n}\}$. Then put a second sequence of $\Cbf$-special subvarieties $(N_m)_m$ with $N_m=M_{n,m_n}$ for $$d_1+\cdots+d_{n-1}< m\leq d_1+\cdots+d_n, \ m_n=m-\sum_{i<n}d_i.$$ $(N_m)_m$ is again a sequence of rigid $\Cbf$-special subvarieties, $\Cbf$-strict as it is already true with $(M_n)_n$. Then the associated sequence of $\Cbf$-special measures $(\mu_{N_m})_m$ converges to $\mu$, hence the averaged sequence $$\nu_n:=\dfrac{1}{d_n}\sum_{\sum_{i<n}d_i<m\leq\sum_{i\leq n}}\mu_{N_m}$$ also converges to $\mu$, from which follows the equidistribution of $\Oscr_n$. \end{proof}

\begin{remark}
If $M$ is pure in the corollary above, then the  number of conjugates of $\Cbf$-special subvarieties is uniformly bounded by a constant independent of the choice of the subdata, as a consequence of \ref{rockmore-tan-bound}  below.

However when $M$ is a rigid Kuga variety, the number of conjugates of $\Cbf$-special subvarieties is not necessarily bounded, see the constructions in \ref{strongly-rigid-bound}. %But we can still divide the set of conjugates of the $M_n$'s into a finite union of sequences $(\Scal_i)_{1\leq i\leq N}$, each $\Scal_i$ being a collection of $\Cbf_i$-special subvarieties for a fixed $\Cbf$.

\end{remark}

Our main theorem is restricted to the case where $M$ admits a triple product decomposition $M=M^\cm\times M^\rig\times S'$ such that $M^\rig$ is strongly rigid in the following sense:

\begin{definition}\label{strong-rigidity} A rigid Kuga datum $(\Pbf,Y)=\Vbf\rtimes_\rho(\Gbf,X)$ (with $\Vbf\neq0$) is said to be strongly rigid, if the connected center of $\Gbf$ acts on $\Vbf$ through a split $\Qbb$-torus.

The notion of strongly rigid subdata, strongly rigid special subvarieties, etc. are defined in the evident way.
\end{definition}

\begin{lemma}{transitivity}
Let $(\Pbf,Y)=\Vbf\rtimes_\rho(\Gbf,X)$ be a strongly rigid Kuga datum, with $\Cbf$ the connected center of $\Gbf$. Then any $\Cbf$-special subdatum of $(\Pbf,Y)$ is strongly rigid.
\end{lemma}

\begin{proof}
We may assume that $\rho:\Gbf\ra\GLbf_\Qbb(\Vbf)$ is faithful. It then suffices to apply \cite{chen-rigid} 2.2.
\end{proof}

\begin{theorem}\label{main-theorem}
Let $(\Pbf,Y)=\Vbf\rtimes(\Gbf,X)$ be a Kuga datum decomposed as a direct product $(\Pbf^\cm,Y^\cm)\times(\Pbf^\rig,Y^\rig)\times(\Gbf',X')$ with $(\Pbf^\rig,Y^\rig)$ strongly rigid, and $M=\Gamma\bsh Y^+$ a connected Kuga variety associated to $(\Pbf,Y)$ which also decomposes into $M^\cm\times M^\rig\times S'$, as in \ref{triple-product-decomposition}. Let $F$ be a number field over which $M$ admits a canonical model, and $\Cbf$ the connected center of $\Gbf$. Assume that $(M_n)_n$ is a sequence of $\Cbf$-special subvarieties in $M$, which is $\Cbf$-strict in the sense that for any $\Cbf$-sepcial subvariety $M'\subsetneq M$ we have $M_n\nsubseteq M'$ for $n$ large enough. Write $\Oscr_n$ for the $\Gal(\Qac/F)$-orbit of $M_n$ in $M$, and $\mu_n$ the average of canonical measures associated to members in $\Oscr_n$. Then $\mu_n$ converges to $\mu$ the canonical measure on $M(\Cbb)$ for the weak-* topology.
\end{theorem}

\begin{corollary}\label{main-corollary}
Let $M$ be a connected Kuga variety defined by the connected Kuga datum $(\Pbf,Y;Y^+)=\Vbf\rtimes(\Gbf,X;X^+)$, and let $\Cbf$ be a $\Qbb$-torus in $\Gbf$. Then for any sequence $(M_n)_n$ of $\Cbf$-special subvarieties in $M$, the Zariski closure of $\bigcup_nM_n$ is a finite union of $\Cbf$-special subvarieties.
\end{corollary}

\begin{proof}
We prove the following equivalent statement: if $Z\subset M$ is a closed subvariety, geometrically irreducible, which equals the Zarski closure of a sequence of $\Cbf$-sepcial subvariety, then $Z$ is a $\Cbf$-special subvariety.

By the same reductions as in \cite{chen-rigid}, one may assume that $M$ is $\Cbf$-sepcial itself. And up to passing to a minimal $\Cbf$-special subvariety containing all the $M_n$'s, one may also assume that $(M_n)_n$ is $\Cbf$-special in $M$; up to passing to a subsequence, one may further assume that $(M_n)_n$ is $\Cbf$-strict.

Take $F$ a field of definition for $M$, over which $M$ admits a canonical model, and splits $\Cbf$. Then by \ref{finite-conjugate}, each member in the $\Gal(\Qac/F)$-conjugates of $M_n$ is $\Cbf$-special, and \ref{main-theorem} implies that the measure associated to $\Oscr_n$ converges to the canonical measure of $M(\Cbf)$ for the weak-* topology, which implies the density of $\bigcup_nM_n$ in $M$. Therefore the union of $\Gal(\Qac/F)$-conjugates of $Z$ is dense in $M$. $Z$ contains a Zariski dense subset of subvarieties defined over $\Qac$, $Z$ is defined over $\Qac$ itself, hence descends to some number field, with finitely many $\Gal(\Qac/F)$-conjugates. $M$ is geometrically irreducible, one thus concludes that all the $\Gal(\Qac/F)$-conjguates of $Z$ coincide and that $Z=M$.% \ref{main-theorem},
\end{proof}

\section{$\Cbf$-special subvarieties and proof of the equidistribution}

The general setting of this section goes as follows:
\begin{assumption}\label{assumption-equidistribution}

We keep the notations as in \ref{main-theorem}: $(\Pbf,Y)$ is a Kuga datum decomposed into a triple product $(\Pbf^\cm,Y^\cm)\times(\Pbf^\rig,Y^\rig)\times(\Gbf'X')$, where $(\Pbf^\cm,Y^\cm)=\Ubf\rtimes(\Tbf,x)$ with $(\Tbf,x)$ toric, $(\Pbf^\rig,Y^\rig)=\Wbf\rtimes_\rho(\Gbf^\rig,X^\rig)$ a strongly rigid Kuga datum with $\Gbf^\rig$ acts on $\Wbf$ faithfully, and $(\Gbf',X')$ a pure Shimura datum. $M$ is a connected Kuga variety associated to $(\Pbf,Y;Y^+)$, which is decomposed into a triple product $M^\cm\times M^\rig\times S'$, with $M^\cm=\Gamma^\cm\bsh Y^\cm$, $M^\rig=\Gamma^\rig\bsh Y^{\rig,+}$, $S'=\Gamma'\bsh X'^+$ connected Kuga varieties associated to each factor in the triple product.

Write $\Cbf=\Tbf\times\Cbf^\rig\times\Cbf'$ for the connected center of $\Gbf$, with $\Cbf^\rig$ the connected center of $\Gbf^\rig$ and $\Cbf'$ that of $\Gbf'$. $(M_n)_n$ is a sequence of $\Cbf$-special subvarieties of $M$, $\Cbf$-strict in the sense of \ref{main-theorem}. By \ref{} 2.5, $M_n$ is given by a $\Cbf$-special subdatum of the form $(\Pbf_n^\cm,Y^\cm_n)\times(\Pbf^\rig_n,Y^\rig_n)\times(\Gbf'_n,X'_n)$, with $(\Pbf^\cm_n,Y^\cm_n)\subset(\Pbf^\cm,Y^\cm)$ $\Tbf$-special, $(\Pbf^\rig_n,Y^\rig_n)\subset(\Pbf^\rig,Y^\rig)$ $\Cbf^\rig$-special, and $(\Gbf'_n,X'_n)\subset(\Gbf',X')$ $\Cbf'$-special. Hence $M_n=M^\cm_n\times M^\rig_n\times S'_n$ with $M^\cm_n$ resp. $M^\rig_n$ resp. $S'_n$ special subvarieties in $M^\cm$ resp. in $M^\rig$ resp. in $S'$ given by the respective factors of $(\Pbf_n,Y_n)$.

The $\Cbf$-strictness of $(M_n)_n$ in $M$ also implies that

(i) $(M^\cm_n)_n$ is $\Tbf$-strict in $M^\cm$;

(ii) $(M^\rig_n)_n$ is $\Cbf^\rig$-special in $M^\rig$;

(iii) $(S'_n)_n$ is $\Cbf'$-special in $S'$.

\textrm{In fact if one of the factor sequence, say, $(S'_n)_n$ fails to be $\Cbf'$-strict, then one finds a $\Cbf'$-special subvariety $S''\subsetneq S'$ that contains infinitely many $S'_n$'s. Then $M'=M^\cm\times M^\rig\times S''\subsetneq M$ is $\Cbf$-special and contains infinitely many $M_n$'s, which contradicts the $\Cbf$-strictness.}

In particular, if we take $F^\cm$ resp. $F^\rig$ resp. $F'$ a field of definition (in the sense of canonical models) for $M^\cm$ resp. for $M^\rig$ resp. for $S'$, and $\Oscr^\cm_n$ resp. $\Oscr^\rig_n$ resp. $\Oscr'_n$ the orbit of $M^\cm_n$ under $\Gal(\Qac/F^\cm)$ resp. of $M^\rig_n$ under $\Gal(\Qac/F^\rig)$ resp. of $S'_n$ under $\Gal(\Qac/F')$, and $\mu^\cm_n$ resp. $\mu^\rig_n$ resp. $\mu'_n$ the average of canonical measures over $\Oscr^\cm_n$ resp. over $\Oscr^\rig_n$ resp. over $\Oscr'_n$, then we have the weak-* limits $\lim_n\mu^\cm_n=\mu^\cm$, $\lim_n\mu^\rig_n=\mu^\rig$, and $\lim_n\mu'_n=\mu'$, with $\mu^\cm$ resp. $\mu^\rig$ resp. $\mu'$ the canonical measure for $M^\cm(\Cbb)$ resp. for $M^\rig(\Cbb)$ resp. for $S'(\Cbb)$.

\end{assumption}

Our approach is to show that for $F$ a suitable number field of definition of $M$, the $\Gal(\Qac/F)$-orbit $\Oscr_n$ of $M_n$ is essentially of the same size as $\Oscr^\cm_n\times\Oscr^\rig_n\times\Oscr'_n$, hence the sequence of averaged canonical measures $(\mu_n)_n$ associated to $M_n$ converges to the same limit as $(\mu^\cm_n\otimes\mu^\rig_n\otimes\mu'_n)_n$, which is exactly the canonical measure on $M$, i.e. $\mu^\cm\otimes\mu^\rig\otimes\mu'$. Here for two measure spaces $(A,\mu_A)$ and $(B,\mu_B)$, we put $\mu_A\otimes\mu_B$ as the product measure on $A\times B$, which sends $A'\times B'$ to $\mu_A(A')\mu_B(B')$, $A'\subset A$ measurable of finite $\mu_A$-volume resp. $B'\subset B$ measurable of finite $\mu_B$-volume.
\bigskip

We start with the pure factor $S'$, for which we first introduce a uniform bound of extensions between reflex fields:

\begin{lemma}\label{rockmore-tan-bound}   Let $(\Gbf,X)$ be a pure Shimura datum with reflex field $E$. Write $r$ for the rank of $\Gbf$, namely the common dimension of maximal $\Qbb$-tori in $\Gbf$.

(1) Then $[E:\Qbb]$ is majorated by  some function $d(r)$ in $r$, independent of the choices of $\Gbf$ and $X$.

(2) When $(\Gbf',X')$ runs through subdata of $(\Gbf,X)$, the degree $[E(\Gbf',X'):E]$ is majorated by some constant $c$, independent of the choice of $(\Gbf',X')$.

\end{lemma}

\begin{proof}
(1) Take a toric subdatum $(\Tbf,x)$ of $(\Gbf,X)$, and extend $\Tbf$ to a maximal $\Qbb$-torus $\Hbf$ of $\Gbf$. Then by \ref{reflex-field} we know that the splitting field $F$ of $\Tbf$ contains $E$.

$F$ is a Galois extension of $\Qbb$, with $\Gal(F/\Qbb)$ isomorphic to the image of $$\Gal(\Qac/\Qbb)\ra\GL_\Zbb(\Xbf_\Hbf)\isom\GL_r(\Zbb),$$ where $\Xbf_\Hbf$ is the $\Gal(\Qac/\Qbb)$-module $\Hom_{\Gr/\Qac}(\Hbf_\Qac,\Gbb_{\mrm\Qac})$. Hence $\Gal(F/\Qbb)$ is isomorphic to a finite subgroup of $\GL_r(\Zbb)$. Apply the uniform upper bound $d(r)$ for the orders of finite subgroups of $\GL_r(\Zbb)$ in \cite{rockmore-tan}, we get $[E:\Qbb]\leq[F:\Qbb]= \#\Gal(F/\Qbb)\leq d(r)$.

(2) As we can embed $\GL_n(\Zbb)$ into $\GL_{n+1}(\Zbb)$ as a subgroup, the bounds $d(*)$ in (1) satisfies $d(n)\leq d(n+1)$ for all $n$. It is clear that $[E(\Gbf',X'):E]\leq[E(\Gbf',X'):\Qbb]$, and it suffices to take $c=d(r)$, as the rank of $\Gbf'$ does not exceeds that of $\Gbf$.\end{proof}

\begin{lemma}\label{pure-bound}
Let $S=\Gamma\bsh X^+$ be a connected Shimura variety defined by a pure Shimura datum $(\Gbf,X)$, with $\Cbf$ the connected center of $\Gbf$, and $S'\subset S$ a $\Cbf$-special subvariety. Let $F$ be a number field of definition for $S$, over which $S$ admits a canonical model. Then the cardinality of the set of $\Gal(\Qac/F)$-conjugates of $S'$ in $S$ is majorated by some constant independent of the choice of $\Cbf$-special subvariety $S'$.
\end{lemma}

\begin{proof}
Up to shrinking $\Gamma$, we may assume that $\Gamma=K\cap\Gbf(\Qbb)_+$ for some \cosg $K\subset\Gbf(\adele)$. Then $S$ is a geometrically connected component of $M_K(\Gbf,X)$, and it is admits a canonical model over a finite abelian extension $E_K$ of $E$ the reflex field of $(\Gbf,X)$, which is given by the formula in \ref{reciprocity} $$\Gal(E^\ab/E_K)=\Ker(\rec_X:\Gal(E^\ab/E)\ra\pitilde(\Tbf)/\nu(K)).$$ Here $\nu:\Gbf\ra\Tbf=\Gbf/\Gbf^\der$ is the abelianization of $\Gbf$, and $\pitilde(\Tbf)/\nu(K)$ is a finite abelian group. Note that $\nu$ also induces an isogeny $\Cbf\ra\Tbf$, whose kernel is the finite $\Qbb$-group $\Hbf:=\Cbf\cap\Gbf^\der$.

Assume that $S'=\wp_\Gamma(X'^+)$ for some $\Cbf$-special subdatum $(\Gbf',X')\subset(\Gbf,X)$, with $X'^+\subset X^+$, and put $K'=K\cap\Gbf'(\adele)$. Then under the morphism $M_{K'}(\Gbf',X')\ra M_K(\Gbf,X)$ defined over the reflex field $E':=E(\Gbf',X')$ with $K'=K\cap\Gbf'(\adele)$, $S'\subset S$ is the image of a geometrically connected component of $M_{K'}(\Gbf',X')$. By the same arguments above, $S'$ is defined over $E'_K$ a finite abelian extension of $E'$, characterized by $$\Gal(E'^\ab/E'_K)=\Ker(\rec_{X'}:\Gal(E'^\ab/E')\ra\pitilde(\Tbf')/\nu'(K'))$$ where $\nu':\Gbf'\ra\Tbf'=\Gbf'/\Gbf'^\der$ is the abelianization, which also defines an isogeny $\Cbf\ra\Tbf'$ of kernel $\Hbf'=\Cbf\cap\Gbf^\der\subset\Hbf$. In particular we have an isogeny $\Tbf'\ra\Tbf$.

Consider the commutative diagram of finite abelian groups $$\xymatrix{\Gal(E'_K/E') \ar[d]^{\Nm_{E'/E}} \ar[r]^\subset & \pitilde{(\Tbf')}/\nu'(K') \ar[d]^{\Tbf'\ra\Tbf} \\
\Gal(E_K/E) \ar[r]^\subset & \pitilde(\Tbf)/\nu(K)}.$$ Here the left vertical arrow $\Gal(E'_K/E')\ra\Gal(E_K/E)$ is induced by $$\Nm_{E'/E}:\Gal(E'^\ab/E)\isom\pitilde(\Gbb_\mrm^{E'})\ra\pitilde({\Gbb_\mrm^E})\isom\Gal(E^\ab/E).$$ Applying \cite{ullmo-yafaev} 2.8, the kernel of the right vertical arrow $\pitilde(\Tbf')/\nu'(K')\ra\pitilde({\Tbf})/\nu(K)$ is majorated by  a constant integer $n(>0)$, which is independent of the choice of $(\Gbf',X')$. Hence $[E'_K:E']\leq n[E_K:E]$. Combined with the uniform upper bound of $[E':E]$, the number of $\Gal(\Qac/E_K)$-conjugates of $S'$ is bounded by $cn[E_K:E]$. It thus suffices to take $F=E_K$.\end{proof}

Then we consider the strongly rigid factor $M^\rig$.

\begin{lemma}\label{strongly-rigid-bound}

Let $M=\Gamma\bsh Y^+$ be the connected Kuga variety associated to $(\Pbf,Y)=\Vbf\rtimes_\rho(\Gbf,X)$, with $\rho:\Gbf\ra\GLbf_\Qbb(\Vbf)$ faithful and rigid, and that $\Cbf$ the connected center of $\Gbf$ splits over $\Qbb$. Assume that $M$ admits a canonical model over some number field $F$. Write $\pi:M\ra S=\pi(\Gamma)\bsh X^+$ for the canonical projection given by $\pi:(\Pbf,Y)\ra(\Gbf,X)$.

(1) If $M'$ is a $\Cbf$-special subvariety given by $(\Pbf',Y')=\Vbf'\rtimes(v\Gbf v^\inv, v\rtimes X)$, where $\Vbf'$ is some subrepresentation of $\Gbf$ in $\Vbf$, and $v\in\Vbf(\Qbb)$, then $\pi(M')=S$, and $M'$ is defined over $F\Qbb^\ab$.

(2) If $M'$ is a $\Cbf$-special subvariety given by a subdatum $(\Pbf',Y')=\Vbf'\rtimes(v\Gbf'v^\inv,v\rtimes X')$ whose image under $\pi$ is $S'\subset S$, with its canonical model over some number field $F'$ ($F\subset F'$), then $M'$ is defined over $F'\Qbb^\ab$.
\end{lemma}

\begin{proof}
Up to shrinking $\Gamma$, we may assume that $\Gamma=K\cap\Pbf(\Qbb)_+$, with $K$ a \cosg of $\Pbf(\adele)$ decomposed as $K=K_\Vbf\rtimes K_\Gbf$ along the rational Levi decomposition $\Pbf=\Vbf\rtimes\Gbf$. It then suffices to prove the lemma for $F=E=E(\Pbf,Y)$ and $F'=E'=E(\Pbf',Y')$.

(1) The abelianization $\Tbf=\Gbf/Gbf^\der$ is a split $\Qbb$-torus as it is isogeneous to $\Cbf$. Then the map $r_X:\Gbb_\mrm^E\ra\Tbf$ thus factors through $\Nm_{E/\Qbb}:\Gbb^E_\mrm\ra\Gbb_\mrm$, because any trivial quotient of $\Xbf_{\Gbb_\mrm^E}$ as a $\Gal(\Qac/\Qbb)$-module factors through $\Xbf_{\Gbb_\mrm^E}\ra\Xbf_{\Gbb_\mrm}$ given by the norm map.

Evaluated over $\pitilde(*)$, we see that $\rec_X:\Gal(E^\ab/E)\isom\pitilde(\Gbb_\mrm^E)\ra\pitilde(\Tbf)$ factors through $\Nm:\pitilde(\Gbb_\mrm^E)\ra\pitilde(\Gbb_\mrm)$, which is the same as the canonical homomorphism $\Gal(E^\ab/E)\ra\Gal(\Qbb^\ab/\Qbb)$ by global class field theory. In particular, the kernel of $\rec_X$ contains $\Gal(E^\ab/E\Qbb^\ab)$.

We have $M'$ a special subvariety defined by $(\Pbf',Y')=\Vbf'\rtimes(v\Gbf v^\inv, v\rtimes X)$, the latter admitting $(v\Gbf v^\inv, v\rtimes X)$ as a pure section. $M'$ is a geometrically connected component of the image of $M_{K'}(\Pbf',Y')\ra M_K(\Pbf,Y)$, where $K'=K\cap\Pbf'(\adele)$. Since $K=K_\Vbf\rtimes K_\Gbf$ and $\Pbf'=\Vbf'\rtimes v\Gbf v^\inv$, the group law formula in \cite{chen-rigid} gives us $K'=K_{\Vbf'}\rtimes vK_\Gbf(v) v^\inv$, where $K_{\Vbf'}=K_\Vbf\cap\Vbf'(\adele)$ and $$K_\Gbf(v)=\{g\in K_\Gbf:v-g(v)\in K_\Vbf\}.$$ The pure section $M_{vK_\Gbf(v)v^\inv}(v\Gbf v^\inv, v\rtimes X)$ of $M_{K'}(\Pbf',Y')$ is isomorphic to $M_{K_\Gbf(v)}(\Gbf,X)$. By \ref{canonical-model}  (3) we have $\pi_0(M_{K'}(\Pbf',Y')_\Qac)\isom\pi_0(M_{vK_\Gbf(v)v^\inv}(\Gbf,X))$, and thus each geometrically connected component of $M_{K'}(\Pbf',Y')$ is defined over $E\Qbb^\ab$.

(2) Since $M'$ is the image of a geometrically connected component of $M_{K'}(\Pbf',Y')$, $M'$ is defined over $E'\Qbb^\ab$ by the same arguments in (1), with $E'=E(\Pbf',Y')$. \end{proof}

Combining the two lemmas above we get:

\begin{corollary}\label{bound-galois}
Let $M=M^\rig\times S'$ be a Kuga variety defined by some Kuga datum $(\Pbf,Y)=\Vbf\rtimes(\Gbf,X)$, whose CM factor is trivial in the triple product decomposition. Let $F$ be a number field of definition over which $M$ admits canonical model, and write $\Cbf$ for the connected center of $\Gbf$. Then there exists a constant $c$ such that each $\Cbf$-special subvariety of $M$ is defined over $F'\Qbb^\ab$ with $F'$ a finite extension of $F$ of degree at most $c$. In fact, by the arguments in \ref{pure-equidistribution}, we can even require $F'$ to be finite Galois extension of $F$ of degree at most $c$.

\end{corollary}

Next we recall some results   on the torsion points of an abelian variety with value in some infinite abelian extensions:

\begin{theorem}\label{ribet-finiteness}[cf.\cite{ribet-finiteness}] Let $A$ be an abelian variety defined over a number field $F$. For $k$ a field extension of $F$, write $A(k)_\tor$ for the torsion subgroup of $A(k)$. Then

(1) $A(F\Qbb^\ab)_\tor$ is finite, where $F\Qbb^\ab$ is the composition of $F$ with $\Qbb^\ab=\bigcup_N\Qbb(\mu_N)$;

(2) if $A$ is a simple CM abelian variety, with complex multiplication by a CM field $E$, then $A(\Qac)_\tor$ is contained in $A(FE^\ab)$.
\end{theorem}

\begin{corollary}\label{uniform-finiteness}
Let $A$ be an abelian variety over a number field $F$.

(1) Fix a constant $c>0$, and let $\Sscr=\{L\}$ be the set of Galois extensions of $F$ of degree at most $c$ contained in a fixed algebraic closure of $F$ (or simply contained in $\Cbb$). Then the torsion part of $\bigcup_\Sscr A(L\Qbb^\ab)$ is finite. In particular there exists a finite extension $F'$ over $F$ such that the torsion part of $\bigcup_\Sscr A(L\Qbb^\ab)$ is contained in the torsion subgroup of $A(F')$.

(2) Let $(a_n)_n$ be a strict sequence of torsion points in $A(\Qac)$. Write $\Oscr'_n$ for the $\Gal(\Qac/F\Qbb^\ab)$-orbits of $a_n$ and $\mu'_n$ the associated averaged Dirac measure. Then $\mu'_n$ converges to the Haar measure on $A(\Cbb)$. 

(2)' The same is true when we replace $\Oscr_n'$ by the $\Gal(\Qac/F_n^\ab)$-orbit of $a_n$, where $(F_n)_n$ is a family of finite Galois extensions of $F$ of degree at most $c$, $c$ being some constant ($>0$).

\end{corollary}

\begin{proof}

(1)  It suffices to prove the existence of such an $F'$ for each simple factor of $A$, and thus we assume for simplicity that $A$ is a simple. We may also enlarge the base field such that $F\supset E$ and that $A(F\Qbb^\ab)_\tor=A[d_0]$ for some $d_0\in\Nbb_{>0}$.
  
  Assume that $\bigcup_{L\in\Sscr}A(L\Qbb^\ab)_\tor$ is infinite.  Then $\forall n\in\Nbb$, there exists $L_n\in\Sscr$ such that $A(L_n\Qbb^\ab)$ contains a torsion point $a_n$ of order $d_n$ such that $d_n< d_{n+1}$ for all $n\in\Nbb$. We take for simplicity that $L_0=F$, and that $d_0|d_n$ for all $n$. We then have $\Gal(\Qac/L_n\Qbb^\ab)$ acting trivially on $<a_n>\subset A(L_n\Qbb^\ab)$ (the torsion subgroup of $A(L_n\Qbb^\ab)$ generated by $a_n$). Since $L_n\Qbb^\ab$ is a Galois extension of the base field $F$, $\Gal(L_n\Qbb^\ab)$ acts trivially on $<\sigma a_n>$ for any $\sigma\in\Gal(\Qac/F)$. Since $A(F\Qbb^\ab)_\tor=A[d_0]$, and $\dfrac{d_n}{d_0}$ tends to $\infty$ as $n\ra\infty$, we thus have $\#\Gal(L_n\Qbb^\ab/F\Qbb^\ab) \ra\infty$ as $n\ra\infty$, which contradicts the estimation $[L_n\Qbb^\ab:F\Qbb^\ab]\leq[L_n:F]\leq c$ $\forall n$.

(2) Take a Drinfeld basis for the total Tate module $\Tbb(A):=\limproj_nA[n]$, so that $A[n]\isom(\dfrac{1}{n}\Zbb/\Zbb)^{2g}$, $g$ being the dimension of $A$ over $F$. Write $\Gbf$ for the Mumford-Tate group of $A$, which is a $\Qbb$-subgroup of $\GSp_{2g}$ for some polarization of $H^1(A(\Cbb),\Qbb)$. Note that $\Gbf$ contains the center $\Gbb_\mrm$ of $\GSp_{2g}$.

Since $A(F\Qbb^\ab)_\tor$ is finite, the action of $\Gal(\Qac/F\Qbb^\ab)$ on $A_\tor=\limind_nA[n]$ factors through $\Gbf(\Zbhat)=\limproj_n\Gbf(\Zbb/n\Zbb)$, and further through $\Gal(\Qac/L)\ra\Gbf(\Zbhat)$ with $\Gal(\Qac/L):=\Ker(\Gal(\Qac/F)\ra\Aut(A(\Qac)_\tor))$. On the other hand, the action of $\Gal(F\Qbb^\ab/F)$ fits into a commutative diagram $$\xymatrix{\Gal(\Qac/F) \ar[d] \ar[r] & \Gbf(\Zbhat) \ar[d] \\ \Gal(F\Qbb^\ab/F) \ar[r] & \Gbb_\mrm(\Zbhat)}$$ where the right vertical arrow is given by $\Gbf\mono\GSp_{2g}\ra\Gbb_\mrm$. As $\Gbb_\mrm$ lifts  to the center of $\GSp_2g$ up to some isogeny, we see that replacing $F$ by a finite extension $F'$ we have $\Gal(F'\Qbb^\ab/F')$ acts on $A_\tor$ through   $\Gbb_\mrm(\Zbhat)$ in the center of $\Gbf(\Zbhat)$, and thus it stabilizes each $A[n]$ and the $\Gal(\Qac/F'\Qbb^\ab)$-orbit of $a_n$ for each $n$. It thus follows that the $\Gal(\Qac/F'\Qbb^\ab)$-orbits of $(a_n)_n$ are equidistributed on $A(\Cbb)$ \wrt the Haar measure. By the same averaging process in \ref{rigid-equidistribution}, the equidistribution also holds for the $\Gal(\Qac/F\Qbb^\ab)$-orbits of the $a_n$'s.

(2)' We take $L$ a finite extension of $F$ such that $\bigcup_nA(F_n\Qbb^\ab)_\tor\subset A(L)_\tor$. Then $F_n\Qbb^\ab\subset L\Qbb^\ab$, and $[F_n\Qbb^\ab:L\Qbb^\ab]$ is uniformly bounded. Since the $\Gal(\Qac/L\Qbb^\ab)$-orbits of the $a_n$'s are equidistributed, the same arguments as in \ref{rigid-equidistribution}   shows that the $\Gal(F_n\Qbb^\ab)$-orbits are equidistributed.  \end{proof}

We now consider the Galois conjugates of $\Cbf$-special subvarieties $(M_n)_n$in a triple product $M=M^\cm\times M^\rig\times S'$ as in \ref{assumption}.

\begin{proof}[Proof of the main theorem]

We are given a $\Cbf$-strict sequence of $\Cbf$-special subvarieties $(M_n)_n$ in $M$, with $M_n=M_n^\cm\times M_n^\rig\times S'_n$, compatible with the triple product decomposition $M=M^\cm\times M^\rig\times S'$.

Write $M^\ncm=M^\rig\times S'$, which is given by the Kuga datum $(\Pbf^\ncm,Y^\ncm)=\Wbf\rtimes(\Gbf^\ncm,X^\ncm)=(\Wbf\rtimes(\Gbf^\rig,X^\rig))\times(\Gbf',X')$. Write $\Cbf^\ncm$ for the connected center of $\Gbf^\ncm$. Then $(M_n^\ncm)_n$ is a $\Cbf^\ncm$-strict sequence of $\Cbf^\ncm$-special subvarieties.

Let $F^\ncm$ be a number field of definition over which $M^\ncm$ admits a canonical model. By \ref{} and \ref{}, each $M_n^\ncm$ is defined over $F_n^\ncm\Qbb^\ab$, where $F_n^\ncm$ is some finite Galois extension over $F^\ncm$ of degree at most $c$, $c$ being some constant only depends on the rank of $\Gbf^\ncm$.

$M^\cm$ is a CM abelian variety, defined by the Kuga datum $(\Pbf^\cm,Y^\cm)=\Ubf\rtimes(\Tbf,x)$, and admits a canonical model over $F^\cm$ a finite abelian extension of the reflex field $E=E(\Tbf,x)$, and each special subvariety of $M^\cm$ is $\Tbf$-special, defined over $E^\ab$. Up to enlarging $F^\cm$, we may also assume that each simple abelian subvariety of $M^\cm_\Qac$ is defined over $F^\cm$.

We take $F=F^\cm F^\ncm$ as the number field of definition of a canonical model of $M=M^\cm\times M^\ncm$. $F_n:=FF_n^\ncm$ is a finite Galois extension of $F$ with degree $\leq c$, and thus $\bigcup_n M^\cm(F_n\Qbb^\ab)_\tor$ is contained in some $M^\cm(L)_\tor$ for a finite extension $L$ of $F$. 

%Take $F'=LF$ as the new field of definition. 
The $\Gal(\Qac/F)$-action on the special subvariety $M_n=M^\cm_n\times M^\ncm_n$ is diveded into two steps:

(i) The action of $\Gal(\Qac/F_n\Qbb^\ab)$ on $M_n$, which fixes the second factor $M_n^\ncm$ as the latter is already fixed by $\Gal(\Qac/F_n\Qbb^\ab)$. In this case write $\Oscr'^\cm_n$ for the $\Gal(\Qac/F_n\Qbb^\ab)$-orbit of $M^\cm_n$, and $\mu_n'^\cm$ for the averaged canonical measure over $\Oscr'^\cm_n$. Then $\mu'^\cm_n$ already converges to the Haar measure of $M^\cm(\Cbb)$, by \ref{uniform-finiteness}(2)'.

(ii) $\Gal(F_n\Qbb^\ab/F)$ on the $\Gal(\Qac/F_n\Qbb^\ab$-orbit of $M_n$, whose effect on $M_n^\ncm$ is the original one mentioned in \ref{bound-galois}. In this case it acts on the factor $M_n^\cm$ through a finite quotient of bounded degree $\leq c$, whose images are torsion subvarieties with translations given by a subset of $M^\cm(L)_\tor$. Since each simple abelian subvariety of $M^\cm$ is already defined over $F$, the canonical measures associated to these $\Gal(F_n\Qbb^\ab/F)$-orbits are of the form $$(\mu'^\cm_n\otimes\mu_{M^\ncm})*(\mu'^\ncm_n)$$ where:

(ii-1) $\mu'^\ncm_n$ is the average of the canonical measures associated to the $\Gal(F_n\Qbb^\ab/F)$-orbit of ${a}\times M_n^\ncm$, with $a$ runs through the torsion points in $M^\cm(L)_\tor$ which translate an abelian subvariety $A_n$ into a geometrically irreducible component of the $\Gal(\Qac/F)$-orbit of $M^\cm_n$;

(ii-2) $*$ is the convolution of measures on $M^\cm\times M^\ncm$, compatible with the translation by $M^\cm(\Cbb)$ on the first factor. 

Applying the same arguments as in \ref{rigid-equidsitribution}, and using the equidistribution on each factor, we get the equidistribution of the $\Gal(\Qac/F)$-orbits of the $M_n$'s with respect the canonical measure on $M(\Cbb)$.\end{proof}

\section{Appendix}

In the Appendix we sketch the proof of some functorial properties of the conjugate of mixed Shimura data and mixed Shimura varieties under an automorphism $\tau\in\Aut(\Cbb/\Qbb)$, which are essentially reductions to the works of Milne and Shih, cf.\cite{milne-shih-taniyama}, \cite{milne-shih-conjugate}, \cite{milne-conjugate}.

We begin by recalling the basic notions of mixed Shimura data and mixed Shimura varieties, following \cite{pink-thesis} 2.1 (see also \cite{chen-rigid} 1.1).

\begin{definition}\label{mixed-shimura-datum-variety}

(1) A mixed Shimura datum is of the form $(\Pbf,\Ubf,Y)$, constructed out of the following data:

(MS-1) $(\Gbf,X)$ a pure Shimura datum in the sense of \cite{deligne-pspm} 2.1.1 ;

(MS-2) A finite-dimensional algebraic representation $\rho_\Vbf:\Gbf\ra\GLbf_\Qbb(\Vbf)$, such that for any $x\in X$, $(\Vbf,\rho_\Vbf\circ x)$ is a rational Hodge structure of type $\{(-1,0),(0,-1)\}$;

(MS-3) A second finite-dimensional algebraic representation $\rho_\Ubf:\Gbf\ra\GLbf_\Qbb(\Ubf)$, such that for any $x\in X$, $(\Ubf,\rho_\Ubf\circ x)$ is a rational Hodge structure of type $(-1,-1)$;

(MS-4) An alternating bilinear map $\psi:\Vbf\times\Vbf\ra\Ubf$, equivariant \wrt the action of $\Gbf$ through $\rho_\Vbf$ and $\rho_\Ubf$; $\psi$ defines a central extension $\Wbf$ of $\Vbf$ by $\Ubf$, and the connected center of $\Gbf$ acts on $\Wbf$ through a $\Qbb$-torus isogeneous to some $\Hbf\times\Gbb_\mrm^d$, with $\Hbf$ a compact $\Qbb$-torus;

(MS-5) $\Pbf=\Wbf\rtimes\Gbf$ for the action of $\Gbf$ on $\Wbf$ as above, and $Y$ is the $\Pbf(\Rbb)\Ubf(\Cbb)$-orbit of $X$ in $\Yfrak(\Pbf)$; here $$\Yfrak(\Pbf):=\Hom_{\Group/\Cbb}(\Sbb,\Pbf_\Cbb)\supset\Yfrak(\Gbf)\supset\Xfrak(\Gbf):=\Hom_{\Group/\Rbb}(\Gbf)\supset X;$$ note that $\Ubf$ is also a normal unipotent $\Qbb$-subgroup of $\Pbf$;

(MS-6) $\Pbf$ is minimal \wrt $Y$ in the sense that if $\Pbf'\subset\Pbf$ then there exists some $y\in Y$ such that $y(\Sbb_\Cbb)\nsubseteq\Pbf'_\Cbb$; note that this implies
the minimality of $\Gbf$ \wrt $X$.

We thus write $(\Pbf,\Ubf,Y)=(\Vbf *_\psi\Ubf)\rtimes_\rho(\Gbf,X)$ to indicate the construction above.

Morphisms between mixed Shimura data are defined in the evident way: they are of the form $(f,f_*):(\Pbf,\Ubf,Y)\ra(\Pbf',\Ubf',Y')$ with $f:\Pbf\ra\Pbf'$ a homomorphism of $\Qbb$-groups sending $\Ubf$ into $\Ubf'$, and the induced map $f_*:\Yfrak(\Pbf)\ra\Yfrak(\Pbf')$ sends $Y$ into $Y'$, equivariant \wrt $f:\Pbf(\Rbb)\Ubf(\Cbb)\ra\Pbf'(\Rbb)\Ubf'(\Cbb)$. Subdata are given by morphisms $(f,f_*)$ such that both $f:\Pbf\ra\Pbf'$ and $f_*:Y\ra Y'$ are injective.

Notions such as products and quotients by a normal $\Qbb$-subgroups are also defined, cf. \cite{chen-rigid} 1.1. In particular, for $(\Pbf,\Ubf,Y)$ constructed as above we have the canonical projection $\pi:(\Pbf,\Ubf,Y)\ra(\Gbf,X)$, which admits pure section, namely $(\Gbf,X)\ra(\Pbf,\Ubf,Y)$ in the way they are defined. All the maximal pure subdata of $(\Pbf,\Ubf,Y)$ are of the form $(w\Gbf w^\inv,wX)$.

The reflex field of $(\Pbf,\Ubf,Y)$ is defined in the same way as \ref{reflex-field}, and equals $E(\Gbf,X)$, argued through the same functorial properties as in \ref{reflex-field}.

(2) For $(\Pbf,\Ubf,Y)$ a mixed Shimura datum and $K\subset\Pbf(\adele)$ a \cosg, we have  the mixed Shimura variety at level $K$, denoted as $M_K(\Pbf,Y)$, which is a quasi-projective reduced variety over $E(\Pbf,Y)$, whose complex poitns are given as $$M_K(\Pbf,Y)(\Cbb)=\Pbf(\Qbb)\bsh[Y\times\Pbf(\adele)/K]$$ Here by the $E(\Pbf,Y)$-rational structure is meant the canonical model of $M_K(\Pbf,Y)$ over $E(\Pbf,Y)$, which is defined in the same manner as in \ref{canonical-model} using toric subdata and the reciprocity law.

With $K$ shrinking, we get the toral mixed Shimura scheme $M(\Pbf,Y)=\limproj_KM_K(\Pbf,Y)$, where the transition maps are defined over $E(\Pbf,Y)$, whose formula over $\Cbb$ are the evident ones. On $M(\Pbf,Y)$ also acts the locally profinite group $\Pbf(\adele)$ via Hecke traslations.

Morphisms between mixed Shimura varieties are also defined in the evident way via morphisms between mixed Shimura data.

And if $K=K_\Wbf\rtimes K_\Gbf$ for \cosgs $K_\Wbf\subset\Wbf(\adele)$ and $K_\Gbf\subset\Gbf(\adele)$, then the canonical projection $M_K(\Pbf,Y)\ra M_{K_\Gbf}(\Gbf,X)$ is geometrically connected, whose fibers are torus bundles over abelian varieties. Thus we often reduce the study of $\pi_0(M_K(\Pbf,Y)_\Qac)$ to the pure case.
\end{definition}

\begin{definition}[cf.\cite{milne-shih-taniyama}, \cite{milne-shih-conjugate}] \label{serre-group}

(1) Let $\Mbf^\circ$ be the category of rational Hodge structures of CM type, i.e. rational Hodge structures $(\Vbf,h)$ with $h(\Sbb)\subset\Tbf_\Rbb$ for some $\Qbb$-torus $\Tbf\subset\GLbf_\Qbb(\Vbf)$. Then $\Mbf$ is a neutral $\Qbb$-linear Tannakian category, whose Tannakian group is $\Sfrak$ a $\Qbb$-protorus, called the (connected) Serre group. It is also equipped with canonical homomorphisms $h:\Sbb\ra\Sfrak_\Rbb$ (the universal Hodge structure of CM type), $w:\Gbb_{\mrm,\Qbb}\ra\Sfrak$ (the universal rational weight), and $\mu:\Gbb_{\mrm\Cbb}\ra\Sfrak_\Cbb$ (the universal cocharacter of Hodge filtration). And when we consider the Tannakian subcategory $\Mbf^L$ of Hodge structure of CM type which splits over a number field $L$, we get a $\Qbb$-torus $\Sfrak^L$, and $\Sfrak=\limproj_L\Sfrak^L$. The composition $\mu:\Gbb_{\mrm\Cbb}\ra\Sfrak_\Cbb\ra\Sfrak^L_\Cbb$ descends to $L$.

(2) Note that $\Mbf^\circ$ is the same as the Tannakian category of motives of absolute Hodge cycles associated to CM abelian varieties over $\Qac$. Let $\Mbf$ be the category of motives for absolute Hodge classes generated by Artin motives over $\Qbb$ and the abelian varieties over $\Qbb$ that are of CM type over $\Qac$. This is a neutral Tannakian category with Betti cohomology as the fiber functor, whose corresponding Tannakian $\Qbb$-group is called the Taniyama group $\Tfrak$, whose neutral connected component is $\Sfrak$, and fits into the following exact sequence  of $\Qbb$-groups  $$1\ra\Sfrak\lra\Tfrak\ot{\pi}\lra\Gal(\Qac/\Qbb)\ra 1$$ where $\Gal(\Qac/\Qbb)$ is viewed as a constant pro-finite $\Qbb$-group, identified canonically with $\pi_0(\Tfrak)$.

Moreover the exact sequence above admits a ''splitting'' over $\adele$ $\spl:\Gal(\Qac/\Qbb)\mono\Tfrak(\adele)$ (i.e. $\pi\circ\spl=\id_{\Gal(\Qac/\Qbb)}$). For each $\tau\in\Gal(\Qac/\Qbb)$, $\Sfrak^\tau:=\pi^\inv(\tau)$ is a right $\Sfrak$-torsor, trivialized by the base point $\spl(\tau)\in\Sfrak^\tau(\adele)$.
%( In more updated literature $\Sfrak$ and $\Tfrak$ are also called the connected Serre group and the Serre group respectively. Here we keep the traditional terminologies used in the references. )

%For example, if $A$ is a CM abelian variety over $\Qbb$, then the Betti cohomology $H^1_B(A(\Cbb),\Qbb)$ is a rational Hodge structure of type $\{(-1,0),(0,-1)\}$, whose Mumford-Tate group is a $\Qbb$-torus $\Tbf$ in $\GSp(\Vbf)$ for a suitable polarization on $\Vbf=H^1_B(A(\Cbb),\Qbb)$. And it corresponds to a unique homomorphism of $\Qbb$-tori $\Sfrak\ra\Tbf$

\end{definition}

From the Serre group we can define conjugates of a mixed Shimura datum by an automorphism of $\Cbb$:

\begin{definition}[cf. \cite{milne-shih-conjugate} Section 4, \cite{milne} Section 6]\label{conjugate-shimura}

(1) Let $(\Gbf,X)$ be a pure Shimura datum, with $(\Tbf,x)$   a toric Shimura datum, and $\tau\in\Aut(\Cbb/\Qbb)$. $(\Tbf,x)$ can be induced from $(\Sfrak,h,\mu)$ through a unique homomorphism $\Sfrak\ra\Tbf$, which induces an action of $\Sfrak$ on $\Gbf$ via $\Sfrak\ra\Gbf$ and the adjoint action. Twist by the right $\Sfrak$-torsor $\spl^\inv(\tau)$ defines a linear $\Qbb$-group $\Gbf^\tau$, which is isomorphic to $\Gbf$ over $\adele$.

Note that the twist by $\spl^\inv(\tau)$ transfers $\Tbf$ into a $\Qbb$-torus $\Tbf^\tau$ in $\Gbf$ which is canonically isomorphic to $\Tbf$ (as the adjoint action is trivial on $\Tbf$).
We put $x^\tau=x\circ \tau:\Sbb\ra\Tbf^\tau$, whose Hodge filtration is given by $\mu_x\circ\tau$, and also put $X^\tau$ to be the $\Gbf^\tau(\Rbb)$-conjugacy class of $x^\tau$. Then $(\Gbf^\tau,X^\tau)$ is a pure Shimura datum, with $(\Tbf^\tau,x^\tau)$ as a toris subdatum. Moreover the reflex field of $(\Gbf^\tau,X^\tau)$ is $\tau(E(\Gbf,X))$.

If we have started with a second toric subdatum $(\Tbf',x')$, then the resulting Shimura datum is isomorphic to $(\Gbf^\tau,X^\tau)$ by a canonical isomorphism. Hence we simply write $(\Gbf^\tau,X^\tau)$ without indicating the choice of toric subdatum $(\Tbf,x)$.

Since the $\Sfrak$-torsors $\spl^\inv(\tau)$ are trivial over $\adele$, we have a canonical isomorphism $\Gbf(\adele)\isom\Gbf^\tau(\adele)$. And we get an isomorphism $\tau(M(\Gbf,X))\ra M(\Gbf^\tau,X^\tau)$ of schemes over $\tau(E(\Gbf,X))$. Take $K\subset\Gbf(\adele)$ \cosg and $K^\tau$ its image in $\Gbf^\tau(\adele)$, we get an isomorphism at finite levels $\tau(M_K(\Gbf,X))\ra M_{K^\tau}(\Gbf^\tau,X^\tau)$.

(2) Let $(\Pbf,\Ubf,Y)$ be a mixed Shimura datum, with $(\Tbf,x)$ a toric subdatum. Then $\Tbf$ extends to a maximal reductive $\Qbb$-subgroup $\Gbf$, hence $(\Tbf,x)$ extends to a maximal pure subdatum $(\Gbf,X)$ of $(\Pbf,\Ubf,Y)$. Then $(\Pbf,\Ubf,Y)$ can be reconstructed from the action of $\Gbf$ on $\Wbf$ the unipotent radical of $\Pbf$, which in turn is determined by two algebraic representations $(\Vbf,\rho_\Vbf)$ and $(\Ubf,\rho_\Ubf)$ plus a $\Gbf$-equivariant bilinear map $\psi:\Vbf\times\Vbf\ra\Ubf$, subject to the conditions in \ref{mixed-shimura-datum}.

Take $\tau\in\Aut(\Cbb/\Qbb)$, we have constructed $(\Gbf^\tau,X^\tau)$ out of $\Sfrak\ra\Tbf\ra\Gbf$ and the $\Sfrak$-torsor $\spl^\inv(\tau)$. Pass further to $\Sfrak\ra\Tbf\ra\Pbf$ through the action of $\Sfrak$ on $\Vbf$ and $\Ubf$ by $\rho_\Vbf$ and $\rho_\Ubf$, we get $\Qbb$-groups $\Pbf^\tau\supset\Wbf^\tau\supset\Ubf^\tau$, and a mixed Shimura datum $(\Pbf^\tau,\Ubf^\tau,Y^\tau)$, where $Y^\tau$ equals the $\Pbf^\tau(\Rbb)\Ubf^\tau(\Cbf)$-orbit of $X^\tau$ in $\Yfrak(\Pbf^\tau)$. $(\Pbf^\tau,\Ubf^\tau,Y^\tau)$ is independent of the choices of toric subdata $(\Tbf,x)$ and of pure sections $(\Gbf,X)$ up to canonical isomorphisms.

We remark that $\Ubf^\tau$ is actually isomorphic to $\Ubf$ as $\Gbf$ acts on $\Ubf$ through a split $\Qbb$-torus, cf.\cite{pink-thesis} 2.14.

Similarly, we have isomorphism of $\tau(E(\Pbf,Y))$-schemes $\tau(M(\Pbf,Y))\isom M(\Pbf^\tau,Y^\tau)$ and isomorphisms at finite levels $\tau(M_K(\Pbf,Y))\isom M_{K^\tau}(\Pbf^\tau,Y^\tau)$, with $K\subset\Pbf(\adele)$ \cosg and $K^\tau$ its image under $\Pbf(\adele)\isom\Pbf^\tau(\adele)$.
\end{definition}

\begin{proposition}\label{functoriality}
(1) Let $(\Gbf,X)$ be a pure Shimura datum, with reflex field $E=E(\Gbf,X)$. Then for $\tau\in\Aut(\Cbb/\Qbb)$, $(\Gbf^\tau,X^\tau)$ is canonically isomorphic to $(\Gbf,X)$ as pure Shimura data.

(2) Let $f:(\Gbf,X)\ra(\Gbf',X')$ be a morphism of pure Shimura data, and take $\tau\in\Aut(\Cbb/\Qbb)$. Then the conjugate by $\tau$ transfers $f$ into $f^\tau:(\Gbf^\tau,X^\tau)\ra(\Gbf'^\tau,X'^\tau)$.

(3) Similar statements hold for mixed Shimura data.
\end{proposition}

\begin{proof}[Sketch of the proof]
(1) See cf.\cite{milne-shih-conjugate} Section 4.

(2) Take $(\Tbf,x)$ a toric subdatum of $(\Gbf,X)$, then its image under $f$ is a toric subdatum $(\Tbf',x'=f\circ x)$ of $(\Gbf',X')$. Conjugate by $\tau$ gives rise to $f^\tau:(\Tbf^\tau,x^\tau)\ra(\Tbf'^\tau,x'^\tau)$. It then suffices to take orbit under the twisted homomorphism $f^\tau:\Gbf^\tau\ra\Gbf'^\tau$ and obtain $f^\tau:(\Gbf^\tau,X^\tau)\ra(\Gbf'^\tau,X'^\tau)$.

Note that the construction of the twisted homomorphism $f^\tau:\Gbf^\tau\ra\Gbf'^\tau$ also uses $f:\Tbf\ra\Tbf'$ and the action of $\Sfrak$ on $\Gbf$ through $\Sfrak\ra\Tbf\ra\Gbf$ resp. on $\Gbf'$ through $\Sfrak\ra\Tbf'\ra\Gbf'$. The resulting morphism $(\Gbf^\tau,X^\tau)\ra(\Gbf'^\tau,X'^\tau)$ only depends on the choice of $(\Tbf,x)\ra(\Tbf',x')$ up to canonical isomorphism.

(3) If $f:(\Pbf,\Ubf,Y)\ra(\Pbf',\Ubf',Y')$ is a morphism of mixed Shimura data, then by choosing $(\Gbf,X)$ a pure section of $(\Pbf,\Ubf,Y)$, we can extend the image $(f(\Gbf),f_*(X))$ into a maximal pure subdatum $(\Gbf',X')$ of $(\Pbf',\Ubf',Y')$, which is a pure section. The morphism $f$ is thus determined by the following data:

(a) a morphism of pure Shimura data $f:(\Gbf,X)\ra(\Gbf',X')$;

(b) algebraic representations $(\Vbf,\rho_\Vbf)$ and $(\Ubf,\rho_\Ubf)$ of $\Gbf$, with an equivariant alternating bilinear map $\psi:\Vbf\times\Vbf\ra\Ubf$, giving rise to the unipotent radical $\Wbf$ with weight filtration characterized by $\Ubf\subset\Wbf$;

(b') similar constructions $(\Vbf',\rho_{\Vbf'})$, $(\Ubf',\rho_{\Ubf'})$, $\psi'$, $\Wbf'$ for $\Gbf'$;

(c) a homomorphism of unipotent $\Qbb$-groups  $\Wbf\ra\Wbf'$, equivariant \wrt $\Gbf\ra\Gbf'$, which is in turn characterized by equivariant maps $\Vbf\ra\Vbf'$, $\Ubf\ra\Ubf'$ compatible with $\psi$ and $\psi'$.

And to construct the conjugate by $\tau$, it suffices to take toric subdata $(\Tbf,x)\subset(\Gbf,X)$, $(\Tbf',x')=(f(\Tbf),f_*(x))\subset(\Gbf',X')$, and repeat the constructions in (2) and \ref{conjugate-shimura}.\end{proof}

\end{document}